\documentclass[11pt,twoside, reqno]{amsart}

\usepackage{amsmath}
\usepackage{amsthm}
\usepackage{amsfonts, amssymb}
\usepackage{mathrsfs}
\usepackage[all]{xy}
\usepackage{url}

\usepackage{latexsym}
\usepackage[dvips]{graphicx}

\theoremstyle{plain}
\newtheorem{lema}{Lemma}
\newtheorem{prop}[lema]{Proposition}
\newtheorem{teo}[lema]{Theorem}

\newtheorem{coro}[lema]{Corollary}
\theoremstyle{remark}

\newtheorem{obs}[lema]{Remark}

\theoremstyle{definition}
\newtheorem{defi}[lema]{Definition}
\newtheorem{ej}[lema]{Example}

\newcommand{\cat}{\textrm{CAT}}

\newcommand{\x}{\mathcal{X}}
\newcommand{\kp}{\mathcal{K}}

\newcommand{\N}{\mathbb{N}}
\newcommand{\R}{\mathbb{R}}
\newcommand{\Z}{\mathbb{Z}}
\newcommand{\Q}{\mathbb{Q}}

\newcommand{\E}{\mathfrak{E}}
\newcommand\restr{\raisebox{-0.3ex}{$|$}\raisebox{0.3ex}{}}
\def\rddots#1{\cdot_{\cdot_{\cdot_{#1}}}}
	
\begin{document}

\title[Cop and robber on finite spaces]{Cop and robber on finite spaces}

\author[J.A. Barmak]{Jonathan Ariel Barmak $^{\dagger}$}

\thanks{$^{\dagger}$ Researcher of CONICET. Partially supported by grants UBACyT 20020190100099BA, CONICET PIP 11220170100357CO, ANPCyT PICT-2017-2806 and ANPCyT PICT-2019-02338.}

\address{Universidad de Buenos Aires. Facultad de Ciencias Exactas y Naturales. Departamento de Matem\'atica. Buenos Aires, Argentina.}

\address{CONICET-Universidad de Buenos Aires. Instituto de Investigaciones Matem\'aticas Luis A. Santal\'o (IMAS). Buenos Aires, Argentina. }

\email{jbarmak@dm.uba.ar}

\begin{abstract}
A cop tries to capture a robber in a topological space $X$ being unable to see him. For which spaces $X$ does the cop have a strategy which allows him to capture the robber independently of his efforts to escape? In other words, when is there a curve $\gamma: \R_{\ge 0}\to X$ which has a coincidence with any other curve in $X$. We analyze in particular the case of finite topological spaces and discover general results and exotic examples about paths in these spaces.
\end{abstract}



\maketitle

\section{Introduction}

Let $X$ be a topological space. By a curve in $X$ we mean a continuous map $\gamma: \R_{\ge 0} \to X$. We say that the cop has a strategy in $X$ if there exists a curve $\gamma$ in $X$ which has a coincidence with any other curve $\rho$ in $X$, that is a point $t\in \R_{\ge 0}$, which depends on $\rho$, such that $\gamma (t)=\rho (t)$. Otherwise we say that the robber has a strategy in $X$. For instance, it is easy to see that the cop has a strategy in the interval $[0,1]\subseteq \R$. He just moves from $0$ to $1$, and then stays still. Also, the existence of a strategy for the cop requires the space to be path-connected and to have the fixed point property: if $f:X\to X$ is fixed point free, then $f\gamma$ is a curve for the robber without coincidences with $\gamma$. Note also that a strategy for the cop is necessarily surjective. 

Different variants of this problem have been studied: (a) If the cop can see the robber at any time, he can change his path depending on the robber moves. Here the values of $\gamma(s)$ for $s\le s_0$ depend on the values of $\rho(s)$ for $s<s_0$. In that case we can define the notion of strategy for the robber in a similar way, which is not exactly the negation of the existence of a strategy for the cop \cite{Bar1}. If the cop has a strategy in our sense, then it has a strategy in this other sense. (b) If the players move in a metric space, restrictions can be put on their maximum speeds and also a positive radius of capture can be allowed, or a radius of visibility. (c) A discrete analogue of the problem can be studied in graphs, in which the robber moves from a vertex to a neighbor while the cop tries to find him by choosing any vertex (or any adjacent vertex) each turn \cite{BW, Has}. (d) There could be more than one pursuer or evader. 

We will see that with a few exceptions, the cop does not have a strategy on polyhedra (Theorem \ref{clascw}). The particular case of finite topological spaces is important for two reasons. Recently, classical dynamical systems have been studied using finite spaces. A real dynamical system can be investigated from a sampled point cloud by taking a cube grid and a multivalued map from the face poset of the grid, considered as a finite space. Conley theory has been developed for finite spaces \cite{BMW} as an alternative for proving the existence of chaos or other phenomena in classical dynamics. In continuous time dynamical systems $\phi : X\times \R_{\ge 0}\to X$, the trajectories $\phi_x:\R_{\ge 0}\to X$ satisfy the following property: if $\phi_x(t)=\phi_x(t')$, then $\phi_x(t+s)=\phi_x(t'+s)$ for every $s\ge 0$. So, if $t<t'$, a ``proper part" of $\phi_x\restr_{[t,\infty)}$ coincides with itself. A continuous time dynamical system in a finite set has necessarily constant trajectories for $t>0$, however nontrivial paths which coincide with proper parts will appear when studying the cop and robber game on finite spaces. On the other hand, our problem when posed for finite spaces, is analogous to a different pursuit-evasion problem involving regular CW-complexes. Namely, suppose that a watcher and a thief move continuously in an art gallery, whose rooms are some of the open cells of a regular cell complex $K$. The watcher cannot see the thief until they are in the same room, moment in which the watcher wins. We will show that the watcher has a strategy if and only if the cop has a strategy in the finite space of rooms ordered by the coface relation.

Path-components of finite spaces are well understood since Stong's paper \cite{Sto}. A $T_0$ topology in a finite set is equivalent to a partial order in the same set, and two points $x,y$ in a finite $T_0$ space $X$ are in the same path-component if and only if there is a sequence (called fence) $x=x_0\le x_1\ge x_2 \le \ldots x_k=y$. However, this does not mean that we understand what a path $\gamma :[0,1]\to X$ in a finite space looks like. It is a key part of this work to study properties of paths and curves in finite spaces and to exhibit examples of wild nature. Many articles about finite spaces deal only with the combinatorial interpretation they have as posets, or even less, with the simplicial complex associated to that poset (order complex). The results in this paper go in a very different direction and the essence of finite spaces is crucial in every proof.  

We have attempted to characterize those finite spaces in which the cop has a strategy. In a first stage it seemed plausible that the existence of such a strategy in a finite $T_0$ space depended only on the subspace of extrema, which consists of minimal and maximal points (Corollary \ref{coroe}). Although we managed to describe the spaces of height 1 in which there is a strategy (Theorem \ref{main1}), this does not give a direct answer to the general question (Example \ref{ye} and Corollary \ref{maincoro}). Some examples (Section \ref{sectionfractal}) which lie between the necessary and sufficient conditions obtained (Corollary \ref{maincoro} and Theorem \ref{main2}), show that further analysis is needed.

\section{General results}

We know already that the cop has a strategy in the 1-dimensional disk. The 2-dimensional case is relevant being the space in which Rado formulated his original problem of the lion and man in 1925, in which players can see each other and have equal maximum speeds \cite{Lit}. We will understand this example and much more in this section.

\begin{obs} \label{obsretract}
If the cop has a strategy in a space $X$, so does he in any retract. Indeed, if $A\subseteq X$ is a retract, with $r:X\to A$ a retraction, and $\gamma: \R_{\ge 0}\to X$ is a strategy for the cop in $X$, then $r\gamma$ is a strategy in $A$, for if $\rho:\R_{\ge 0} \to A$ is another curve, then there exists $t\in \R_{\ge 0}$ such that $\gamma(t)=i \rho (t)$ ($i:A\to X$ the inclusion) and then $r\gamma (t)=\rho (t)$. 
\end{obs}

A \textit{subdivision} of $\R_{\ge 0}$ is a collection of closed intervals of $\R_{\ge 0}$ such that: no member of the collection is a singleton, the collection covers $\R_{\ge 0}$, the intersection of two different members of the collection is empty or a singleton, and the collection is locally finite (for every $t\in \R_{\ge 0}$ there is a neighborhood intersecting only finitely many members). This last condition is equivalent to requiring that each $t\in \R_{\ge 0}$ is either in the interior of one interval or in the boundary of exactly two intervals. Note that a subdivision of $\R_{\ge 0}$ contains a unique unbounded interval of the form $[a,\infty)$ or otherwise all its members are bounded. A subdivision of a compact real interval $[a,b]$ is defined in the same way. The last condition implies that the collection in this case is finite. In this second case our definition is the standard definition of subdivision.

Let $X$ be a topological space, $\mathcal{U}$ an open cover of $X$, and $\gamma:\R_{\ge 0} \to X$ a curve (or $\gamma:[a,b]\to X$ a path). A subdivision of $\R_{\ge 0}$ (or $[a,b]$) is said to be \textit{$\gamma$-admissible} if each interval is contained in the preimage $\gamma^{-1}(U)$ of some $U\in \mathcal{U}$. Every time we have a $\gamma$-admissible subdivision of $\R_{\ge 0}$ (or $[a,b]$) we will choose for each interval $I$ an open set $U\in \mathcal{U}$ containing $\gamma(I)$ (this may be done in more than one way). A $\gamma$-admissible subdivision is \textit{minimal} if there are no adjacent intervals which are contained in the same $\gamma^{-1}(U)$.

By a Lebesgue number argument it is easy to prove that for any space $X$, any open cover $\mathcal{U}$ and any path $\gamma:[a,b]\to X$, there exists a $\gamma$-admissible subdivision of $[a,b]$. By gluing adjacent intervals, this can be turned into a minimal $\gamma$-admissible subdivision. If we now consider a curve $\gamma':\R_{\ge 0}\to X$, then there exists a minimal $\gamma'$-admissible subdivision of $\R_{\ge0}$, but this requires a more delicate argument. The details can be found in Theorem \ref{admissible} in the appendix. In most of the results in this article we will only need the path version.
 
\begin{lema} \label{lemacuatroabiertos}
Let $X$ be a topological space and let $\mathcal{U}=\{A_0,A_1,A_2,B\}$ be an open cover of $X$ such that $A_0,A_1,A_2$ are pairwise disjoint. Moreover, suppose there exist $a_i\in A_i\smallsetminus B$ for each $i=0,1,2$, $b\in B$ and a path $\omega_i$ in $X$ from $b$ to $a_i$ whose image does not intersect $A_{j}$ for $j\neq i$. Then the robber has a strategy in $X$.
\end{lema}
\begin{proof}
Let $\gamma :\R_{\ge 0} \to X$ be a curve. We must show there is another curve $\rho$ having no coincidence with $\gamma$. By Theorem \ref{admissible} there is a minimal $\gamma$-admissible subdivision of $\R_{\ge 0}$. Since the subdivision is locally finite, a function $\R_{\ge 0} \to X$ is continuous if and only if it is continuous in each interval. Assume $[t_0,t_1]$, $[t_1,t_2]$, $[t_2,t_3]$ are three adjacent intervals in the subdivision and suppose $\gamma$ maps $[t_0,t_1]$ to some $A_i$. Since the $A_j$ are pairwise disjoint, by minimality of the subdivision $\gamma([t_1,t_2])\subseteq B$ and $\gamma([t_2,t_3])\subseteq A_j$ for some $j$ which could be equal to $i$ or not. Let $k\in \{0,1,2\}$ be different from $i,j$. Let $s_0=\frac{t_0+t_1}{2}$ be the center of the first interval, and $s_2$ the center of the third one. We define $\rho$ in $[s_0,s_2]$ as follows: $\rho\restr_{[s_0,t_1]}$ is (a reparametrization of) $\omega_k$, $\rho\restr_{[t_1,t_2]}$ is the constant path at $a_k$ and $\rho\restr_{[t_2,s_2]}$ is (a reparametrization of) $\overline{\omega}_k$, the inverse of $\omega_k$. See Figure \ref{figcuatroabiertos}.

\begin{figure}[h] 
\begin{center}
\includegraphics[scale=0.55]{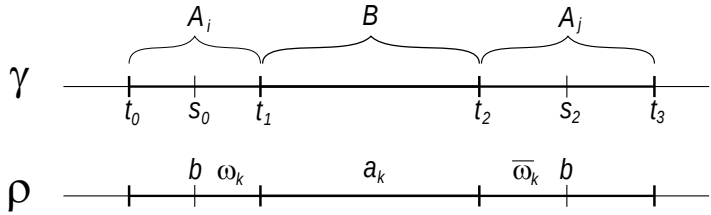}
\caption{A curve avoiding $\gamma$ in Lemma \ref{lemacuatroabiertos}.}\label{figcuatroabiertos}
\end{center}
\end{figure}

Thus $\rho\restr_{[s_0,s_2]}$ is a continuous loop at $b$. By hypothesis $\rho\restr_{[s_0,s_2]}$ and $\gamma\restr_{[s_0,s_2]}$ are coincidence-free. We define $\rho$ in this way for each triple of intervals as above. If the subdivision of $\R_{\ge 0}$ contains no unbounded interval and the first interval $I_1=[0,t_1]$ is mapped by $\gamma$ into $A_i$, then $\rho$ can be defined in the first half of $I_1$ as the inverse of the path already defined in the second half (or just as a constant path). This map is well-defined, continuous, and has no coincidence with $\gamma$. If there are no unbounded intervals, but the first interval $I_1$ is mapped to $B$, while the second $I_2$ is mapped to $A_i$, then we define $\rho$ in the first half of $I_2$ as the inverse of the path defined in the second half, and we define $\rho$ to be constant in $I_1$. In the case that there is an unbounded interval $[t, \infty)$ in the subdivision (in particular when the subdivision has just one or two members), $\rho$ can be defined with similar ideas, and details are left to the reader.  
\end{proof}

\begin{ej} \label{ejy}

Let $Y$ be a union of three compact real intervals $[a_0,b_0]$, $[a_1,b_1]$, $[a_2,b_2]$ identifying $b_0,b_1,b_2$ (see Figure \ref{figcapacitor}).

\begin{figure}[h] 
\begin{center}
\includegraphics[scale=0.55]{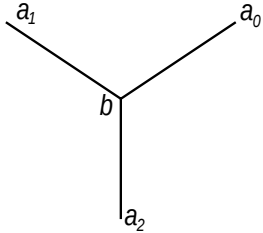}
\caption{The space $Y$.}\label{figcapacitor}
\end{center}
\end{figure}

Define for each $i=0,1,2$, $A_i=[a_i,b_i)$, and $B=Y\smallsetminus \{a_0,a_1,a_2\}$. Let $\omega_i$ be a path from $b_i$ to $a_i$ in $[a_i,b_i]$. Then Lemma \ref{lemacuatroabiertos} applies, so the robber has a strategy in $Y$.  

\end{ej}

In particular the robber has a strategy in the 2-dimensional disk $D^2$, since the space $Y$ is a retract. If we impose a bound for the maximum speed of the robber, and not for the cop, then even when the robber is invisible, the cop has a strategy in $D^2$ \cite[Problema 7]{olimpiada}. We have the following classification.

\begin{teo} \label{clascw}
Let $K$ be a CW-complex. Then the cop has a strategy in $K$ if and only if $K$ is homeomorphic to the interval $[0,1]$ or it is the singleton.
\end{teo} 
\begin{proof}
If $\dim (K)\ge 2$, the space $Y$ in Example \ref{ejy} is a retract of $K$, so the robber has a strategy by Remark \ref{obsretract}. If $\dim (K)\le 1$ and $K$ has a cycle, then $S^1$ is a retract, and since it lacks the fixed point property, the robber has a strategy in $K$. If $K$ is disconnected, the robber has a strategy. If $K$ is a tree with a vertex of degree greater than or equal to $3$, then $Y$ is a retract. If $K$ is a tree with every vertex of degree smaller than or equal to $2$, it is homeomorphic to $(0,1)$, $[0,1)$, $[0,1]$ or the singleton. In the first two cases there is a fixed point free map, and in the latter two the cop has a strategy.
\end{proof}

Recall that a space $X$ satisfies the separation axiom $T_0$ if for any pair of different points there exists an open set containing exactly one of them.

\begin{prop}
If a topological space is not $T_0$, then the robber has a strategy.
\end{prop}
\begin{proof}
Let $X$ be a space with two different points $x,y$ such that any open set containing one of them also contains the other. Then the map $X\to X$ which maps $x$ to $y$ and all the remaining points to $x$, is continuous and fixed point free.
\end{proof}

The cop can have a strategy in non-compact spaces as the next example shows.

\begin{ej} \label{ejnocompacto}
Let $X=\Z_{\ge 0}$ with the topology generated by the basis $\{\{0\}\}\cup \{\{0,n\} : n\ge 1\}$. Then $X$ is non-compact as the basis defined admits no finite subcover. Define $\gamma: \R_{\ge 0} \to X$ by $\gamma (n)=n$ for every $n\in \Z_{\ge 0}$ and $\gamma (t)=0$ if $t\notin \Z_{\ge 0}$. Then $\gamma$ is continuous and it is a strategy for the cop. If $\rho$ is a curve in $X$ which passes through $0$, then $\rho^{-1}(0) \subseteq \R_{\ge 0}$ is open and non-empty. Since $\gamma^{-1}(0)$ is dense, $\rho$ and $\gamma$ have a coincidence point. If $\rho$ does not go through $0$, then it is constant as $\Z_{\ge 1}$ is a discrete subspace of $X$. Thus $\rho$ and $\gamma$ have a coincidence also in this case.
\end{ej}

 \label{defstrong}
We say that the cop has a \textit{strong strategy} in a space $X$ if there exists a path $\gamma:[0,1]\to X$ which has a coincidence with every other path in $X$. The path $\gamma$ (and any linear reparametrization $[t,t']\to X$) is then called a \textit{strong strategy} in $X$. In this case the curve $\R_{\ge 0}\to X$ which coincides with $\gamma$ in $[0,1]$ and is constant in $\R_{\ge 1}$ is a (regular) strategy for the cop. On the other hand, the cop may have a strategy in a space $X$ and not a strong strategy. Indeed, a necessary condition for the existence of a strong strategy is compactness, so we can take as $X$ the space in Example \ref{ejnocompacto}.

Note that if the cop has a strong strategy in a Hausdorff space $X$, then $X$ is necessarily metrizable by the Hahn-Mazurkiewicz Theorem.

Example \ref{ejnocompacto} can be better understood in the language of finite (or more generally, Alexandroff) spaces. Finite spaces are rarely Hausdorff. In fact, a finite $T_1$ space is always discrete. Finite $T_0$ spaces are relevant in Homotopy Theory and Dynamical Systems. The rest of the paper will be focused on the cop and robber problem for finite spaces. Our first result is very elementary and does not require any previous knowledge about these spaces.

\begin{prop} \label{lomismo}
Let $X$ be a finite topological space. If the cop has a strategy in $X$, then it also has a strong strategy in $X$.
\end{prop}
\begin{proof}
Assume there is no strong strategy. Let $\gamma: \R_{\ge 0}\to X$ be a curve in $X$. We prove that there exists another curve in $X$ having no coincidence with $\gamma$.

For every $n\ge 1$, let $\gamma_n$ be the restriction of $\gamma$ to $[0,n]$. By assumption, for every $n$ there exists a path $\rho_n:[0,n]\to X$ having no coincidence with $\gamma_n$. Since $X$ is finite, there exists a subsequence $(\rho_{n_{k_1}}(1))_{k_1\in \N}$ of $(\rho_{n}(1))_{n\ge 1}$ which is constant. There is a subsequence $(\rho_{n_{{k_1}_{k_2}}}(2))_{k_2\in \N}$ of $(\rho_{n_{k_1}}(2))_{k_1\ge 2}$ which is constant, and so on. We define $\rho:\R_{\ge 0}\to X$ to be equal to $\rho_{n_{{k_1}_{{k_2}_{\rddots {{k_m}_1}}}}}$ in the interval $[m,m+1]$ for each $m\ge 0$. Clearly $\rho$ is well-defined, continuous and has no coincidence with $\gamma$.
\end{proof}

\section{Finite spaces, the lifting Theorem and the art gallery}

If $X$ is a finite poset, there is a topology in $X$ given by the down-sets, i.e. subsets $U\subseteq X$ which satisfy that $x\le y$ and $y\in U$ implies $x\in U$. This topology is $T_0$. Conversely, if $X$ is a finite $T_0$ space, then for every $x\in X$ we define the minimal open set $U_x$ as the intersection of all the open sets containing $x$. The minimal open sets are open and in fact they are a basis for the topology. We define an order in $X$ by $x\le y$ if $x\in U_y$. These two maps which associate a topology to every order and an order to each topology are mutually inverse \cite{Ale}. In view of this correspondence we consider finite $T_0$ spaces and finite posets as the same thing. A map $f:X\to Y$ between finite $T_0$ spaces is continuous if and only if it is order-preserving, that is $x\le x'$ implies $f(x)\le f(x')$. In finite spaces connected and path-connected components coincide. Two points $x,y$ in a finite $T_0$ space are in the same component if and only if there exists a fence $x=x_0\le x_1\ge \ldots x_n=y$ from $x$ to $y$.

Finite $T_0$ spaces will be represented by their Hasse diagram, a digraph whose vertices are the points of the space and with an arrow from $x$ to $y$ if $x$ is covered by $y$, i.e. $x<y$ and there is no $x<z<y$. In the graphical representation of Hasse diagrams all arrows are assumed to point upwards.

Since for a point $x$ in a finite space $X$, $U_x=\{y\in X | \ y\le x\}$ is the smallest open set containing $x$, a sequence $(x_n)_{n\in \N}$ converges to a $x$ if and only if for some $n_0\in \N$, $x_n\le x$ for every $n\ge n_0$.

Given a finite $T_0$ space $X$, we denote by $\kp(X)$ its order complex. This is the simplicial complex whose simplices are the non-empty chains of $X$. Let $\mu : \kp(X)\to X$ be the map which maps every point in the interior of the simplex $x_0<x_1<\ldots<x_n$ to $x_0$. Then $\mu$ is continuous and moreover it is a weak homotopy equivalence (\cite[Theorem 2]{Mcc} and \cite[Corollary 3.7]{Wof}). If $K$ is a finite simplicial complex, we denote by $\x(K)$ the poset of simplices of $K$ ordered by inclusion (face poset). Since $\kp (\x (K))$ is the first barycentric subdivision of $K$, there is also a weak homotopy equivalence $K\to \x(K)$.

If $X$ is a finite $T_0$ space, we denote by $X^{op}$ the space with the same underlying set and whose open sets are the closed sets of $X$. The order of the poset $X^{op}$ is the opposite of the order in $X$.

Recall that a space $X$ is said to be perfectly normal if for any disjoint closed sets $F,H\subseteq X$, there exists a map $f:X\to [0,1]$ such that $f^{-1}(0)=F$ and $f^{-1}(1)=H$. For instance all metric spaces satisfy this axiom. In \cite[Theorem 3.5]{Wof} Wofsey proved a result which has the following particular case.

\begin{teo}[Wofsey's lifting theorem] \label{wofsey}
Let $X$ be a finie $T_0$ space, $Y$ a perfectly normal topological space and $f:Y\to X$ a continuous map. Then there exists $\widetilde{f}:Y\to \kp(X)$ such that $\mu \widetilde{f}=f$.
\end{teo}

We have the following immediate consequence.

\begin{coro}
Let $X$ be a finite $T_0$ space such that $\kp(X)$ is a space in which the cop has a strategy. Then the cop has a strategy in $X$ as well.
\end{coro}
\begin{proof}
Let $\gamma:\R_{\ge 0}\to \kp(X)$ be a strategy for the cop. We claim that $\mu \gamma$ is a strategy for the cop in $X$. Indeed, if $\rho:\R_{\ge 0}\to X$ is a curve, by Theorem \ref{wofsey} there exists $\widetilde{\rho}:\R_{\ge 0}\to \kp(X)$ such that $\mu \widetilde{\rho}=\rho$. Thus, there exists $t\in \R_{\ge 0}$ such that $\gamma(t)=\widetilde{\rho}(t)$, and so $\mu \gamma (t)=\rho (t)$.
\end{proof}

\begin{obs} \label{obsintervalos}
Unfortunately Theorem \ref{clascw} says that not in many CW-complexes the cop has a strategy. Still, this corollary will be useful for the cases when $\kp(X)$ is homeomorphic to $[0,1]$, as in Figure \ref{fintervalos}.
 
\begin{figure}[h] 
\begin{center}
\includegraphics[scale=0.35]{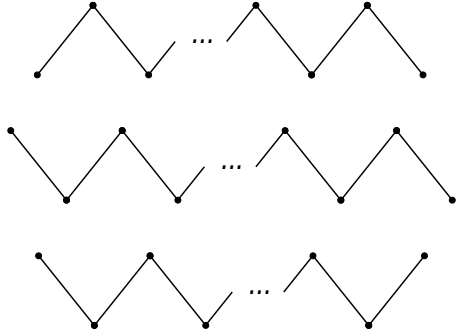}
\caption{Spaces with associated complex homeomorphic to $[0,1]$, in which the cop has a strategy.}\label{fintervalos}
\end{center}
\end{figure}

In fact in these spaces any path $\gamma$ from one leaf in the Hasse diagram to the other leaf is a strong strategy for the cop. A lift $\widetilde{\gamma}:[0,1]\to \kp(X)$ has a coincidence with any other path in $\kp(X)$ by Bolzano's Theorem, so $\gamma=\mu \widetilde{\gamma}$ has a coincidence with any path in $X$.
\end{obs}

There is a relationship between the cop and robber game in finite spaces with another problem, already described in the introduction, played in polyhedra and some subspaces.

There is a thief hidden in an art gallery. When the doors of the gallery close, the thief starts walking around the different rooms. At the same time a watcher begins his route. If at some point both persons are in the same room, then the watcher automatically sees the thief, turns the alarm on, and the thief is caught. The watcher has no idea where the thief is until they are in the same room. If the thief manages to stay in a different room during the whole night, then the doors of the gallery open and he is free. Is there a route the watcher can make every night, so he can find the thief before morning, regardless the thief starting position and the way he moves? Or the thief can always remain unseen during the whole night? We will assume the gallery $S$ is a collection of open cells of a regular CW-complex (with the subspace topology), each cell being a room. So there are rooms of different dimension, and the thief can get caught in any of them. A strategy for the watcher is a path $\gamma: [0,1]\to S$ such that for any other path $\rho:[0,1]\to S$ there exists $t$ with $\gamma(t)$ and $\rho(t)$ being in the same open cell.

We extend the definition of face poset to $S$, so $\x (S)$ is the poset of open cells ordered by the face relation. 

\begin{prop}
Let $S$ be a collection of open cells of a finite regular CW-complex $K$. Then the watcher has a strategy in $S$ if and only if the cop has a strategy in the finite space $\x (S)^{op}$.
\end{prop}
\begin{proof}
Let $\varsigma: S \to \x(S)^{op}$ be the map which maps every point in the open cell $e\subseteq S$ to $e\in \x (S)^{op}$. It is easy to see that $\varsigma$ is continuous (see \cite[Theorem 11.3.2]{Bar}). The polyhedron $\kp (\x(S))$ is a subspace of $S$ via the restriction $\iota: \kp(\x(S))\to S$ of the usual homeomorphism $K'\to K$. If $e_0<e_1<\ldots<e_n$ is a chain in $\x(S)$, the map $\iota$ maps the open simplex $\{e_i\}_{0\le i\le n}$ to the open cell $e_n$ (see \cite[Ch. III, Theorem 1.7]{LW}). In other words, $\varsigma \iota=\mu :\kp (\x(S)^{op})\to \x(S)^{op}$. The rest of the proof is a direct application of the lifting Theorem. 

Assume first that $\gamma:[0,1]\to S$ is a strategy for the watcher. We claim that $\varsigma \gamma$ is a strong strategy for the cop. Let $\rho:[0,1]\to \x(S)^{op}$ be any path. By Theorem \ref{wofsey}, there exists $\widetilde{\rho}:[0,1]\to \kp(\x(S))$ such that $\mu \widetilde{\rho}=\rho$. Then $\iota \widetilde{\rho}$ is a path in $S$ and there exists $t\in [0,1]$ such that $\iota \widetilde{\rho}(t)$ and $\gamma(t)$ are in the same open cell. Thus $\rho(t)=\mu \widetilde{\rho}(t)=\varsigma \iota \widetilde{\rho}(t)=\varsigma \gamma (t)$.  

Conversely, suppose now that the cop has a strategy in $\x(S)^{op}$, so there exists a strong strategy $\gamma:[0,1]\to \x(S)^{op}$ by Proposition \ref{lomismo}. There exists $\widetilde{\gamma}:[0,1]\to \kp(\x(S))$ such that $\mu \widetilde{\gamma}=\gamma$. We show that $\iota \widetilde{\gamma}$ is a strategy for the watcher. Let $\rho:[0,1]\to S$ be any path. Then $\varsigma \rho$ must have a coincidence with $\gamma$, say at $t\in [0,1]$. Thus, $\varsigma \rho(t)=\gamma (t)=\mu \widetilde{\gamma}(t)=\varsigma \iota \widetilde{\gamma}(t)$. This means that $\rho(t)$ and $\iota \widetilde{\gamma}(t)$ lie in the same open cell.
\end{proof}

\begin{ej}
In Figure \ref{figmuseos} we see three art galleries. The 0-dimensional rooms are colored with red, the 1-dimensional with blue and the 2-dimensional with yellow. In the first one there are two rooms of dimension 2, three of dimension 1 and one of dimension 0. Here the watcher has a strategy (see Section \ref{sectionfractal}). In the second gallery there are seven rooms of dimension 2, two of dimension 1 and two of dimension 0. Here the watcher has a strategy as well (see Theorem \ref{main1}). In the third gallery there is one room of dimension 2, one of dimension 1 and two of dimension 0. Here the watcher does not have a strategy (see Example \ref{ye}).
\begin{figure}[h] 
\begin{center}
\includegraphics[scale=0.85]{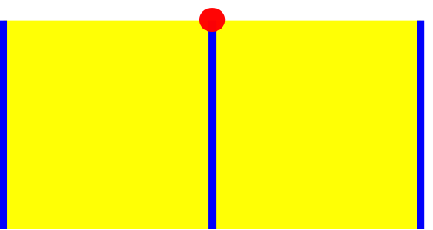}
\qquad
\includegraphics[scale=0.85]{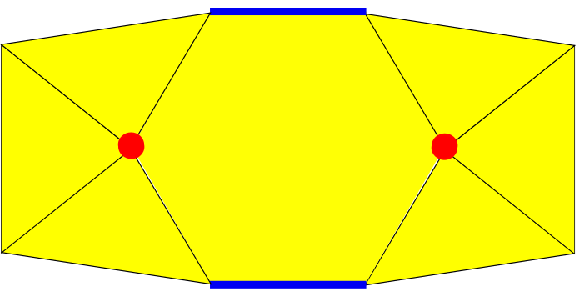}
\qquad
\includegraphics[scale=0.85]{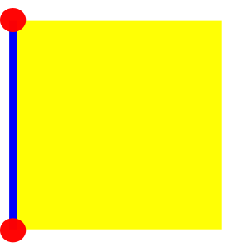}
\caption{Two galleries in which the watcher has a strategy and another in which the thief has a strategy.}\label{figmuseos}
\end{center}
\end{figure}
\end{ej}

%\begin{teo}
%Let $K$ be a finite regular CW-complex, $A\subseteq \R$ a subspace, $\gamma:A\to \x(K)^{op}$ a continuous map. Then there exists a continuous map $\omega: A\to K$ such that $p\omega=\gamma$. 
%\end{teo}
%
%\begin{proof}
%
%\end{proof}

\section{The subspace of extrema}

Let $X$ be a finite $T_0$ space. We denote by $\E (X)$ the subspace of extrema of $X$, formed by its maximal and its minimal points. The aim of this section is to prove that if the cop has a strategy in $X$, so does he in $\E (X)$ (Corollary \ref{coroe}). An alternative proof is discussed in Remark \ref{kevin}. In general $\E (X)$ is not a retract of $X$. However we have the following

\begin{prop} \label{prope}
Let $X$ be a finite $T_0$ space and let $\gamma: \R_{\ge 0} \to X$ be a curve. Then there exists a curve $\E(\gamma): \R_{\ge 0} \to \E(X)$ such that for every $t\in \R_{\ge 0}$, $\gamma(t)\in \E(X)$ implies $\E(\gamma)(t)=\gamma(t)$.
\end{prop}
\begin{proof}
The proof will be by induction in the cardinality of $X\smallsetminus \E(X)$. Assume $a\in X$ is a point which is not maximal nor minimal and it is covered only by maximal points. We are going to define a curve $\widetilde{\gamma}:\R_ {\ge 0} \to X\smallsetminus \{a\}$ which coincides with $\gamma$ at every $t\in \R_{\ge 0}$ such that $\gamma(t)\neq a$. Then the proof concludes by the induction assumption. In order to define $\widetilde{\gamma}$, we consider the connected components of $\gamma^{-1}(a)$. They are intervals of type O: $(t,t')$ (with $t'\in \R_{>0}\cup \{\infty\}$), and intervals of type not O: $[t,t']$ (possibly just one point), $[t,t')$ ($t'\in \R_{>0}\cup \{\infty\})$, $(t,t']$. The new map $\widetilde{\gamma}$ will be constant in each of these intervals. 

Let $\hat{F}_a$ be the set of (maximal) elements in $X$ which are greater than $a$. Let $b$ be a minimal point of $X$ which is smaller than $a$. In each interval of type O, $\widetilde{\gamma}$ is defined to be constant with value $b$. The definition of $\widetilde{\gamma}$ in the intervals of type not O will be given later.

 %If $[t,t')$ is of type $3$ and $t'\neq \infty$, then by continuity of $\gamma$, $\gamma(t')\in \hat{F}_a$. Indeed $\gamma(t')$ lies in the closure of $\gamma([t,t'))=\{a\}$, and it is not $a$. Define $\widetilde{\gamma}|_{[t,t')}$ to be constant $\gamma(t')$. The case that $[t,t')$ is of type 3 with $t'=\infty$ will be treated later. If $(t,t']$ is of type 4, then $\widetilde{\gamma}|_{(t,t']}$ is the constant $\gamma(t)\in \hat{F}_a$. Finally, the most delicate case is when $[t,t']$ is of type 1. In that case $\widetilde{\gamma}|_{[t,t']}$ will also be constant $x$ for some $x\in \hat{F}_a$, but we have to choose it carefully in order that $\widetilde{\gamma}$ is continuous.

Let $t<t'\in \R_{\ge 0}$ be such that $\gamma(t)=x$ and $\gamma(t')=y$ are two different elements in $\hat{F}_a$ and that $(t,t')\cap \gamma^{-1}(\hat{F}_a)=\emptyset$. Then $[t,t']$ is called a change-interval and we will choose a point $t''$ in $(t,t')$ as follows. If $\gamma\restr_{(t,t')}$ is not the constant map to $a$, we take $t''$ so that $\gamma(t'')\neq a$. Otherwise we take $t''$ to be any point in $(t,t')$. Note that in this second case $(t,t')$ is an interval of type O. The point $t''$ is called a right margin for $x$ and a left margin for $y$ (see Figure \ref{fige}). Note that if $t_0<t_1<t_2$ are margins (left and right), then there exists $t\in (t_0,t_1)$ and $t'\in (t_1,t_2)$ with $\gamma(t),\gamma(t')$ different elements of $\hat{F}_a$. Since $\hat{F}_a$ is discrete and a closed subspace of $X$, the set of margins is discrete and closed.

If $\gamma(t)=x\in \hat{F}_a$ for some $t\in \R_{\ge 0}$ then either there is no margin smaller than $t$ or the greatest margin smaller than $t$ is a left margin for $x$. Symmetrically, either there is no margin greater than $t$ or the smallest margin greater than $t$ is a right margin for $x$. 
Suppose first that the image of the curve $\gamma$ contains at least two different elements of $\hat{F}_a$. Then for each $x\in \hat{F}_a$ let $G_x$ be the union of the following open intervals: 1. All the intervals $(t_0,t_1)$ such that $t_0$ is a left margin for $x$, $t_1$ is a right margin for $x$, and there is no margin in the interval. 2. If the smallest margin $t$ in $\R_{\ge 0}$ is a right margin for $x$, then we also add the interval $[0,t)$ to $G_x$. 3. If there is a greatest margin $t'$ in $\R_{\ge 0}$, which is a left margin for $x$, then we add $(t',\infty)$ to $G_x$. 

If there is only one element $x\in \hat{F}_a$ in the image of $\gamma$, then there are no margins and we define $G_x=\R_{\ge 0}$ while $G_y=\emptyset$ for all $x\neq y\in \hat{F}_a$. If $\textrm{Im}(\gamma)\cap \hat{F}_a=\emptyset$, we choose an element $x\in \hat{F}_a$ and set $G_x=\R_{\ge 0}$, $G_y=\emptyset$ for all $y\neq x$.

Note that the open sets $G_x$ for $x\in \hat{F}_a$, are pairwise disjoint and that $\gamma^{-1}(x)\subseteq G_x$ for every $x\in \hat{F}_a$. Also, each interval of type not O is contained in a $G_x$. Indeed, none of these intervals contains a margin, the intervals are connected, and the $G_x$ cover all of $\R_{\ge 0}$ with exception of the margins.

We are now ready to define $\widetilde{\gamma}$ in the intervals of type not O. If $I$ is an interval of type not O, then it is contained in $G_x$ for a unique $x\in \hat{F}_a$. We define $\widetilde{\gamma}\restr_{I}$ to be the constant map to $x$.

\begin{figure}[h] 
\begin{center}
\includegraphics[scale=0.65]{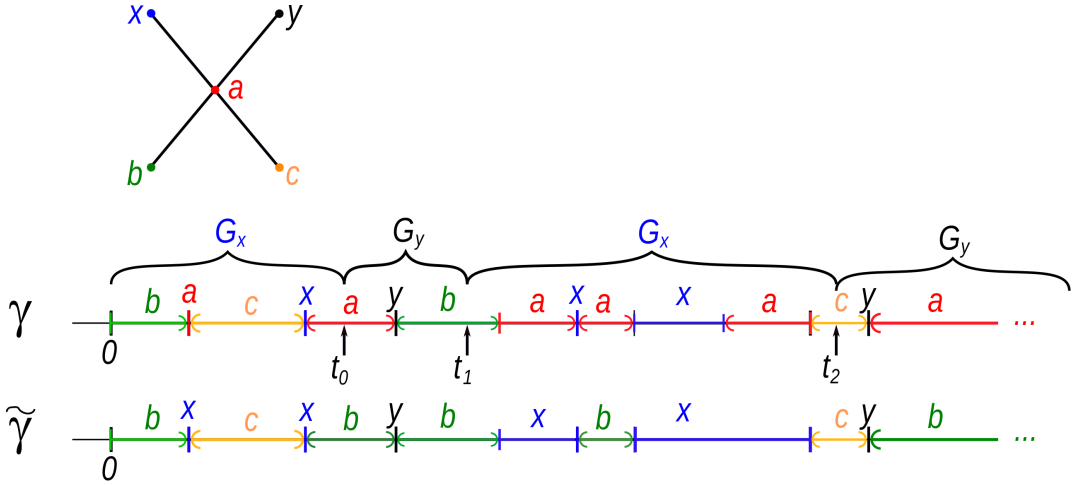}
\caption{A space $X$ of five points, a curve $\gamma$ in $X$, the margins $t_0,t_1,t_2$, the sets $G_x,G_y$, and the curve $\widetilde{\gamma}$ constructed in the proof of Proposition \ref{prope}.}\label{fige}
\end{center}
\end{figure}

We claim that the function $\widetilde{\gamma}:\R_{\ge 0}\to X\smallsetminus \{a\}\subseteq X$ is continuous. Let $(t_n)_{n\in \N}$ be a sequence in $\R_{\ge 0}$ which converges to some $t$. We want to prove that $\widetilde{\gamma}(t_n)\to \widetilde{\gamma}(t)$.

Case 1i: $\gamma(t)\neq a$ and eventually $\gamma(t_n)\neq a$ (that is for some $n_0\in \N$ this happens for every $n\ge n_0$). In this case $\widetilde{\gamma}(t)=\gamma(t)$ and eventually $\widetilde{\gamma}(t_n)=\gamma(t_n)$, so by the continuity of $\gamma$, $\widetilde{\gamma}(t_n)\to \widetilde{\gamma}(t)$.

Case 1ii: $\gamma(t)\neq a$ and $\gamma(t_n)=a$ for infinitely many values of $n$. By assumption $a\to \gamma(t)$, so $\gamma(t)\ge a$, and then $\gamma(t)=x$ for some $x\in \hat{F}_a$. Then $t\in G_x$ and since $G_x$ is open, $t_n$ lies eventually in $G_x$. Therefore eventually $\widetilde{\gamma}(t_n)$ is equal to $x$ (if $t_n$ is in an interval of type not O) or smaller than or equal to $\gamma(t_n)$ (if $t_n$ is in an interval of type O or if $\gamma(t_n)\neq a$). Therefore $\widetilde{\gamma}(t_n)\to x=\widetilde{\gamma}(t)$.

Case 2: $\gamma(t)=a$. If $t$ is in an interval of type O, then $t_n$ lies in the same interval eventually, so $\widetilde{\gamma}(t_n)=\widetilde{\gamma}(t)$ eventually and the convergence is trivial. Otherwise, $t$ lies in an interval of type not O. Thus, there exists $x\in \hat{F}_a$ such that $t\in G_x$ and $\widetilde{\gamma}(t)=x>a=\gamma(t)$. Just as in Case 1ii, $t_n$ lies eventually in $G_x$, so eventually $\widetilde{\gamma}(t_n)$ is equal to $x$ or smaller than or equal to $\gamma(t_n)$. Thus $\widetilde{\gamma}(t_n)\to \widetilde{\gamma}(t)$.

\end{proof}

Of course, the curve $\E(\gamma)$ is not uniquely determined by $\gamma$ in general. 

\begin{coro} \label{coroe}
Let $X$ be a finite $T_0$ space. If the robber has a strategy in $\E(X)$, then so does he in $X$.
\end{coro}
\begin{proof}
Let $\gamma:\R_{\ge 0}\to X$ be a curve in $X$. By Proposition \ref{prope} there is a curve $\E(\gamma):\R_{\ge 0}\to \E(X)$ which coincides with $\gamma$ in $\gamma^{-1}(\E(X))$. By hypothesis there exists a curve $\rho:\R_{\ge 0}\to \E(X)\subseteq X$ having no coincidences with $\E(\gamma)$. Then $\rho$ is a curve in $X$ which has no coincidences with $\gamma$.
\end{proof}

\section{Strategies and obstructions to their existence}

Motivated by Corollary \ref{coroe} we study those finite $T_0$ spaces of height $1$ (i.e. the longest chain has two elements) in which the robber has a strategy. We begin with examples which turn out to be crucial in the classification.

\begin{ej} \label{ejnhc}
Let $m\in \N$ and let $X$ be the non-Hausdorff cone over the discrete space of $m$ points. In other words, it is the space $X$ of Figure \ref{fignhc}. We claim that the cop has a (strong) strategy in $X$. Let $\gamma:[0,1] \to X$ be the path defined by $\gamma(0)=a$, $\gamma(\frac{1}{n})=a$ for every $n\in \N$, and in the interval $(\frac{1}{n+1}, \frac{1}{n})$, $\gamma$ is constant with value $c_{\overline{n}}$ for every $n\in \N$, where $\overline{n}$ denotes class modulo $m$.

\begin{figure}[h] 
\begin{center}
\includegraphics[scale=0.55]{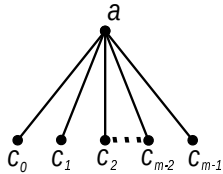}
\caption{The space in Example \ref{ejnhc}.}\label{fignhc}
\end{center}
\end{figure}

It is easy to see that $\gamma$ is indeed continuous. Let $\rho:[0,1] \to X$ be a path. If $\rho(0)=a$, then $\rho(0)=\gamma(0)$. Otherwise $\rho(0)=c_k$ for some $0\le k\le m-1$. Since $\{c_k\}$ is open in $X$, $\rho^{-1}(c_k)$ is an open neighborhood of $0$ and then it intersects an interval of the form $(\frac{1}{n+1}, \frac{1}{n})$ for some $n\in \N$ with $\overline{n}=k$. Thus $\rho$ and $\gamma$ also have a coincidence point in this case. This proves that $\gamma$ is a strong strategy for the cop. Note that if $U\subseteq [0,1]$ is any neighborhood of $0$ and $\rho:U\to X$ is a map, then $\rho$ has a coincidence with $\gamma\restr_U$. This property will be used later and it will be generalized in Theorem \ref{teomax}.
\end{ej}

%The example $S_{1,m}^l$ is maximal in some sense. If we increase the number of tails, then the robber is the one with a strategy.%zzz

\begin{ej} \label{escorpion21}
Let $X=S_{2,1}$ be the space depicted in Figure \ref{fescorpion21}. We claim that the robber has a strategy in $X$.

\begin{figure}[h] 
\begin{center}
\includegraphics[scale=0.55]{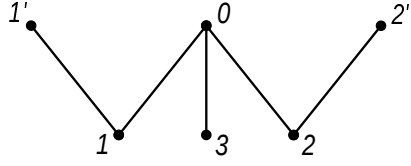}
\caption{The scorpion $S_{2,1}$.}\label{fescorpion21}
\end{center}
\end{figure}

Let $\gamma: [0,1]\to X$ be a path. By Proposition \ref{lomismo} it suffices to show that there is a path $\rho:[0,1]\to X$ which has no coincidence with $\gamma$ (alternativelly we can work with curves $\R_{\ge 0}\to X$ and use Theorem \ref{admissible}). Consider the open cover $\mathcal{U}=\{U_0, U_{1'},U_{2'}\}$. There exists a minimal $\gamma$-admissible subdivision of $[0,1]$. As always we choose for each interval in the subdivision an element of $\mathcal{U}$ whose preimage contains the interval. If $[t_0,t_1]$ and $[t_1,t_2]$ are adjacent intervals of the subdivision, then by minimality and the fact that $U_{1'}\cap U_{2'}=\emptyset$, $U_0$ must have been chosen for one and only one of the intervals. 

In the intervals of the subdivision for which $U_0$ was not chosen, we define $\rho$ to be the constant map to $3$. Since $3\notin U_{i'}$ for $i=1,2$, $\rho$ has no coincidences with $\gamma$ in those intervals.

Let $[t_1,t_2]$ be an interval of the subdivision for which $U_0$ has been chosen. Suppose $[t_0,t_1]$ and $[t_2,t_3]$ are intervals of the subdivision adjacent to the first one. Let $U_{i'}$ be the member of $\mathcal{U}$ chosen for the first interval and $U_{j'}$ for the second. Since $U_{i'}\cap U_0=\{i\}$, $\gamma(t_1)=i$ and similarly $\gamma(t_2)=j$. We will define $\rho\restr_{[t_1,t_2]}:[t_1,t_2]\to X$. If $3\notin \gamma([t_1,t_2])$, we define $\rho\restr_{[t_1,t_2]}$ to be the constant map to $3$. Otherwise, since $\{3\}\in X$ is open, there exist $t_1<s_1<s_2<t_2$ such that $\gamma((s_1,s_2))=\{3\}$. Also, since $\{i\}$ and $\{j\}$ are open, there exist $s_0>t_1$ and $s_3<t_2$ such that $\gamma([t_1,s_0))=\{i\}$ and $\gamma((s_3,t_2])=\{j\}$. Note that $s_0<s_1$ and $s_2<s_3$ (see Figure \ref{figdoscolas}).

\begin{figure}[h] 
\begin{center}
\includegraphics[scale=0.55]{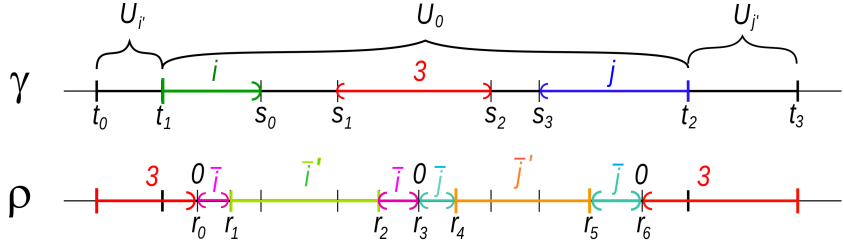}
\caption{The robber has a strategy in $S_{2,1}$.}\label{figdoscolas}
\end{center}
\end{figure}

 Let $r_0,r_1,\ldots, r_6\in [0,1]$ be such that $t_1<r_0<r_1<s_0<s_1<r_2<r_3<r_4<s_2<s_3<r_5<r_6<t_2$. Let $\overline{i}$ be the element of $\{1,2\}$ different from $i$, and define $\overline{j}$ similarly. Define $\rho$ to be $3$ in $[t_1,r_0)\cup(r_6,t_2]$, $0$ in $\{r_0,r_3, r_6\}$, $\overline{i}$ in $(r_0,r_1)\cup (r_2,r_3)$, $\overline{i}'$ in $[r_1,r_2]$, $\overline{j}$ in $(r_3,r_4)\cup (r_5,r_6)$ and $\overline{j}'$ in $[r_4,r_5]$. It is clear then that $\rho\restr_{[t_1,t_2]}$ is continuous and has no coincidence with $\gamma\restr_{[t_1,t_2]}$. There are two cases that we have not contemplated yet. When the interval for which $U_0$ has been chosen is the first or when it is the last in the subdivision. In these cases the definition of $\rho$ is simpler and left to the reader. The path $\rho:[0,1]\to X$ is well-defined, continuous and has no coincidence with $\gamma$.
\end{ej}

%The next example shows, among other things, that reversing the order in a finite space may change the fact that a given player has a strategy.

\begin{ej} \label{ejnhco}
Let $X$ be the opposite of the non-Hausdorff cone of the discrete space in $m+1$ points (see Figure \ref{fignhco}). We claim that the cop has a (strong) strategy in $X$. This is very similar to Example \ref{ejnocompacto}. We define $\gamma:[0,1]\to X$ by $\gamma(\frac{k}{m})=b_k$ for $0\le k\le m$, and the value of $\gamma$ is $a$ in all the other points. Let $\rho:[0,1]\to X$ be another path. If $a$ is in the image of $\rho$, then since $\{a\}\subseteq X$ is open and $\gamma^{-1}(a)$ is dense in $[0,1]$, $\rho$ and $\gamma$ have a coincidence point. Otherwise, $\rho$ is constant, and since $\gamma$ is surjective, there is also a coincidence. Note that $\gamma^{-1}(b_k)$ consists of just one point.

\begin{figure}[h] 
\begin{center}
\includegraphics[scale=0.55]{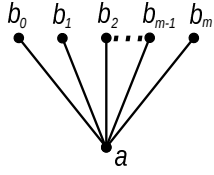}
\caption{The space $X$ in Example \ref{ejnhco}.}\label{fignhco}
\end{center}
\end{figure}
\end{ej}

\begin{ej} \label{escorpion30op}
Let $X=S_{3,0}^{op}$ be the space in Figure \ref{figescorpion30o}. Then the robber has a strategy in $X$.
\begin{figure}[h] 
\begin{center}
\includegraphics[scale=0.55]{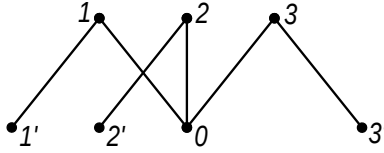}
\caption{The scorpion $S_{3,0}^{op}$.}\label{figescorpion30o}
\end{center}
\end{figure}

Let $\mathcal{U}=\{U_1,U_2,U_3\}$ and let $\gamma:\R_{\ge 0} \to X$ be a curve (alternatively we could work with a path). There is a minimal $\gamma$-admissible subdivision of $\R_{\ge 0}$, and for each interval $I$ we choose a member of $\mathcal{U}$ containing $\gamma(I)$. An interval $I$ of the subdivision is called special if $\gamma(I)\cap \{1',2',3'\}\neq \emptyset$. If $I=[t_0,t_1]$ (or $I=[t_0,\infty)$) is a special interval and $i'\in \gamma(I)$ for some $i=1,2,3$, then there is an open interval $(s_0,s_1)$ contained in $I\cap \gamma^{-1}(i')$. We choose one such interval for each special interval. %Moreover we choose $r_0,r_1\in \R$ with $s_0<r_0<r_1<s_1$.

Suppose now $[t_0,t_1]$ and $[t_2,t_3]$ are two special intervals with $t_1<t_2$. Moreover, suppose there is no special interval contained in $[t_1,t_2]$. Let $(s_0,s_1)\subseteq \gamma^{-1}(i')$ be the interval chosen inside $[t_0,t_1]$. Similarly $(s_2,s_3)\subseteq \gamma^{-1}(j')$ is the interval chosen for $[t_2,t_3]$. Here $i,j\in \{1,2,3\}$ could be equal or not. Let $k\in \{1,2,3\}$ be different from $i,j$. We define $\rho\restr_{[s_1,s_2]}$ to be the constant function $k'$. Note that $\rho\restr_{[s_1,s_2]}$ has no coincidence with $\gamma\restr_{[s_1,s_2]}$ since there are no special intervals in $[t_1,t_2]$. We do this for every pair of consecutive special intervals as above. 

Let $[t_0,t_1]$ be a special interval. Suppose $(s_0,s_1)$ is the interval chosen inside $[t_0,t_1]\cap \gamma^{-1}(i')$, and that $\rho(s_0)=k', \rho(s_1)=l'$ have already been defined for $k,l\neq i$. We define $\rho\restr_{[s_0,s_1]}$ to be a path from $k'$ to $l'$ not touching $i'$ (in fact there is a path going through $k',k,0,l,l'$). Thus, $\rho\restr_{[s_0,s_1]}$ and $\gamma\restr_{[s_0,s_1]}$ are coincidence-free. This defines a curve $\rho$ avoiding $\gamma$. We only have to pay attention to the first special interval and all the intervals at its left, and the last special interval if any. Also, there could be no special intervals at all. All these cases are simple and details are left to the reader.      
\end{ej}

\section{The classification in height 1 and an open question in the general case} \label{seccion1}

If $K$ is a compact polyhedron with nontrivial $b_1(K)=H_1(K;\Q)$, then by a result of Borsuk, $S^1$ is a retract, and in particular $K$ lacks the fixed point property. For finite spaces this is not true: a finite space $X$ with the fixed point property may have $b_1(X)\neq 0$ \cite{Bar3}. For finite spaces of height 1, though, an analogue of Borsuk's result holds.

\begin{prop} \label{ciclo}
Let $X$ be a finite $T_0$ space of height 1 which contains a cycle $x_0<x_1>x_2<\ldots <x_{n-1}>x_n=x_0$ ($n\ge 4$, $x_i\neq x_j$ if $0\le i<j\le n-1$). Then the robber has a strategy in $X$.
\end{prop}
\begin{proof}
We can assume $X$ is connected and the length $n$ of the cycle is minimal. We claim that this cycle $C$ is a retract of $X$. This follows almost directly from a well-known result in graph theory \cite[Proposition 2.51]{HN}. Concretely, let $G$ be the underlying undirected graph of the directed graph obtained from the Hasse diagram of $X$ by removing the edge $(x_0,x_{n-1})$. Define the map $r:X\to X$ by $r(x)=x_i$, where $i$ is the minimum between $n-1$ and the distance between $x$ and $x_0$ in $G$. Then $r$ is a continuous retraction from $X$ to $C$. Since $C$ admits a fixed point free map, then so does $X$. Thus the robber has a strategy in $X$.
%Let $\gamma: \R_{\ge 0}\to X$ be a curve. Let $C=\{x_0,x_1,\ldots ,x_{n-1}\} \subseteq X$. Since $\kp (C)$ is a cycle in the graph $\kp (X)$, there exists a retraction $r:\kp (X)\to \kp(C)$. Let $\varphi: \kp(C)\to \kp(C)$ be the simplicial map which maps $x_i$ to $x_{i+2}$, subindices taken modulo $n$. Let $i:\kp(C)\to \kp (X)$ denote the inclusion. By Wofsey's lifting theorem (Theorem \ref{wofsey}) there exists a lifting $\widetilde{\gamma}:\R_{\ge 0}\to \kp(X)$ of $f$. Let $\rho=\mu \varphi i r \widetilde{\gamma}: \R_{\ge 0} \to X$. Suppose there exists $t\in \R_{\ge 0}$ such that $\gamma(t)=\rho(t)$. 
\end{proof}

Note that for a connected finite $T_0$ space of height 1, having a cycle is equivalent to not being simply connected and also to having nontrivial first Betti number. 

If a finite $T_0$ space $X$ of height $1$ has no cycle, then any non-empty connected subspace is a retract: just map every point in $X$ (in the component of the subspace if not connected) to the closest point in the subspace in the Hasse diagram. In particular, if $X$ has a subspace $A$ homeomorphic to $S_{2,1}$ or $S_{3,0}^{op}$, then the robber has a strategy. Conversely, if $X$ is of height $1$ with no cycles and $S_{2,1}$ and $S_{3,0}^{op}$ are not subspaces, then the cop has a strategy.

\begin{teo} \label{main1}
Let $X$ be a connected finite $T_0$ space of height $1$. Then the cop has a strategy (equivalently a strong strategy) in $X$ if and only if the following conditions are satisfied simultaneously:

\begin{itemize}
\item $X$ contains no cycle.
\item No subspace of $X$ is homeomorphic to $S_{2,1}$ nor $S_{3,0}^{op}$.
\end{itemize}
\end{teo}
\begin{proof}
We have already established the necessity of the conditions. Conversely, suppose the conditions are satisfied.
Let $a_0,a_1,\ldots, a_l$ be a fence of maximal length in $X$. That is, these $l+1$ points are pairwise different, $a_i$ and $a_{i+1}$ are comparable in $X$ for $0\le i \le l-1$ and there is no longer fence in $X$. We can assume $l\ge 2$. Since $X$ contains no cycle, $a_i$ and $a_j$ are not comparable for $|i-j|>1$. Suppose $a_i$ is minimal in $X$. If $i=0$, the unique point of $X$ comparable with $a_i$ is $a_1$, by maximality of the fence. If $i=l$, $a_i$ is only comparable with $a_{i-1}$. If $0<i<l$, let $\{a_{i-1},a_{i+1}, b_{i,1},b_{i,2},\ldots, b_{i,m_i}\}$ be the set of points in $X$ comparable with (=greater than) $a_i$. Since $S_{3,0}^{op}$ is not a subspace of $X$, $X$ contains no cycle, and the length of the fence is maximal, then the points $b_{i,j}$ are only comparable with $a_i$. Conversely, suppose $0\le i\le l$ is such that $a_i$ is maximal. Again, if $i=0$ or $i=l$, $a_i$ is only comparable with one point in $X$. If $2\le i \le l-2$, since $S_{2,1}$ is not a subspace of $X$, the unique points comparable with (=smaller than) $a_i$ in $X$ are $a_{i-1}$ and $a_{i+1}$. Let $\{a_0,a_2, c_1,c_2,\ldots ,c_{m_1}\}$ be the set of points comparable with $a_1$. Let $\{a_{l-2},a_l, d_1,d_2,\ldots, d_{m_{l-1}}\}$ be the set of points which are smaller than $a_{l-1}$. By maximality of the length of the fence, the points $c_j$ are only comparable with $a_1$ and the points $d_j$ only comparable with $a_{l-1}$. Since $X$ is connected, there are no other points in $X$ than the $a_i$, the $b_{i,j}$, the $c_j$ and the $d_j$. 

If $a_0$ is maximal, we can define a new space by adding a point $a_{-1}$ smaller than $a_0$ and incomparable with any other point in $X$. Similarly, if $a_l$ is maximal, we can add a point $a_{l+1}$ smaller than $a_l$ and incomparable with any other point. The new space satisfies the hypothesis of the statement and contains $X$ as a retract. Thus, we can assume $a_0$ and $a_l$ are both minimal (so $l$ is even), and $X$ looks as the diagram in Figure \ref{figgeneral}. The case $l=2$ follows directly from Example \ref{ejnhc}, so we will assume $l\ge 4$.

\begin{figure}[h] 
\begin{center}
\includegraphics[scale=0.55]{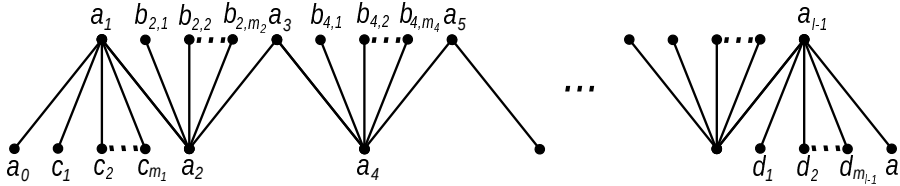}
\caption{The space $X$ in Theorem \ref{main1}.}\label{figgeneral}
\end{center}
\end{figure}

Essentially we glue the strategies in Examples \ref{ejnhc} and \ref{ejnhco} (see Remark \ref{horizontal} for comments in this direction). Let $\gamma: [0,l] \to X$ be the path defined as in Figure \ref{figaltura1}. Namely $\gamma(2i+1)=a_{2i+1}$ for $0\le i\le \frac{l}{2}-1$. The interval $[2i-1,2i+1]$ (for $1\le i \le \frac{l}{2}-1$) is subdivided in $m_{2i}+1$ intervals by considering points $2i-1<t_{2i,1}<t_{2i,2}<\ldots<t_{2i,m_{2i}}<2i+1$. We define $\gamma(t_{2i,j})=b_{2i,j}$, while $\gamma(t)=a_{2i}$ for the remaining points in $(2i-1,2i+1)$. In the interval $[0,1)$, $\gamma$ is defined as follows. The value of $\gamma$ in $t_k=1-\frac{1}{k}$ is $a_1$ for every $k$, while $\gamma$ in $(t_k,t_{k+1})$ is the constant map to $c_{\overline{k}}$, where $\overline{k}$ is the class of $k$ modulo $m_1+2$ and $c_0=a_0$, $c_{m_1+1}=a_2$. In the interval $(l-1,l]$, $\gamma$ is defined with a symmetric idea by taking the value $a_{l-1}$ in $s_k=l-1+\frac{1}{k}$ for every $k$, while $\gamma$ in $(s_{k+1},s_k)$ is constant with value $d_{\overline{k}}$, being $\overline{k}$ the class modulo $m_{l-1}+2$ and $d_0=a_l$, $d_{m_{l-1}+1}=a_{l-2}$. Clearly $\gamma$ is continuous.

\begin{figure}[h] 
\begin{center}
\includegraphics[scale=0.7]{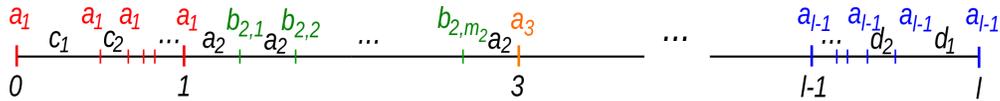}
\caption{The strong strategy $\gamma$.}\label{figaltura1}
\end{center}
\end{figure}

We claim $\gamma$ is a strong strategy for the cop. Let $\rho:[0,l] \to X$ be another path. We want to prove that $\gamma$ and $\rho$ have a coincidence point. Let $Y$ be the subspace of $X$ which consists of the points $a_1,a_2,\ldots, a_{l-1}$ and let $r:X\to Y$ be the retraction that maps $a_0,c_1,c_2,\ldots, c_{m_1}$ to $a_1$, it maps $a_l,d_1,d_2,\ldots, d_{m_{l-1}}$ to $a_{l-1}$, and for each even $2\le i \le l-2$, it maps $b_{i,1},b_{i,2},\ldots, b_{i,m_i}$ to $a_i$.

Let $\omega :[1,l-1]\to Y$ be the map defined by $\omega (2i+1)=a_{2i+1}$ for $0\le i\le \frac{l}{2}-1$ while $\omega$ in $(2i-1,2i+1)$ is constant with value $a_{2i}$ for $1\le i\le \frac{l}{2}-1$. Note that $\omega$ has a coincidence with any other path $[1,l-1]\to Y$ by Remark \ref{obsintervalos}. In particular, there exists $t\in [1,l-1]$ such that $\omega(t)=r\rho\restr_{[1,l-1]}(t)$. 

Case 1: $t=1$. In this case $r\rho(1)=\omega(1)=a_1$. This implies that $\rho(1)\in \{a_1,a_0,c_1,c_2,\ldots,$ $c_{m_1}\} \subseteq U_{a_1}$. Thus, there is a neighborhood $U$ of $1$ in $[0,1]$ such that $\rho (U)\subseteq U_{a_1}$. By Example \ref{ejnhc}, $\rho$ and $\gamma$ have a coincidence in $U$.

Case 2: $t=l-1$ is symmetric to case 1.

Case 3: $t=2i+1$ for some $1\le i\le \frac{l}{2}-2$. In this case $r\rho(2i+1)=\omega(2i+1)=a_{2i+1}$, so $\rho(2i+1)=a_{2i+1}=\gamma(2i+1)$.

Case 4: $t\in (2i-1,2i+1)$ for some $1\le i \le \frac{l}{2}-1$. In this case $r\rho(t)=\omega (t)=a_{2i}$, so $\rho (t)\in \{a_{2i},b_{2i,1},b_{2i,2},\ldots, b_{2i,m_{2i}}\}$. If $\rho (t)=a_{2i}$, since $\{a_{2i}\}$ is open and $\gamma\restr_{[2i-1,2i+1]}^{-1}(a_{2i})$ is dense in $[2i-1,2i+1]$, $\rho$ and $\gamma$ have a coincidence point. Suppose then that $\rho(t)=b_{2i,j}$ for some $1\le j\le m_{2i}$. If $\rho\restr_{(2i-1,2i+1)}$ is constant, then $\rho(t_{2i,j})=b_{2i,j}=\gamma(t_{2i,j})$. If $\rho\restr_{(2i-1,2i+1)}$ is not constant, necessarily $a_{2i}$ is in its image. And by the reasoning at the beginning of this case, $\rho$ and $\gamma$ have a coincidence. 
\end{proof}

\begin{obs} \label{horizontal}
A slightly different proof of Theorem \ref{main1} may be given using a ``horizontal" retraction instead of a ``vertical" one. The proof would be by induction on $l$ and would use the following result, whose proof is an easy exercise:

Let $Z$ be a topological space and let $X,Y$ be two subspaces of $Z$ such that $Z=X\cup Y$. Let $z_0\in X\cap Y$, and suppose there exists a retraction $r_X:Z\to X$ such that $r_X(Y\smallsetminus X)=\{z_0\}$. Suppose $\gamma_X:[0,1]\to X$ and $\gamma_Y:[0,1]\to Y$ are strong strategies for the cop in $X$ and $Y$ respectively and that $\gamma_X^{-1}(z_0)=\{1\}$ while $\gamma_Y(0)=z_0$. Furthermore, suppose one of the following holds:

\noindent i) There exists a retraction $r_Y:Z\to Y$ such that $r_Y(X\smallsetminus Y)=\{z_0\}$ and $\gamma_Y^{-1}(z_0)=\{0\}$, or

\noindent ii) $Y$ is an open subspace of $Z$ and for any neighborhood $U\subseteq [0,1]$ of $0$ and any map $\rho:U\to Y$, $\rho$ and $\gamma_Y\restr_U$ have a coincidence point. 

Then $\gamma=\gamma_X*\gamma_Y:[0,1]\to Z$ is a strong strategy for the cop in $Z$.
\end{obs}
%\begin{proof}
%Let $\rho:[0,1]\to X\vee Y$ be a path. Let $r_X:X\vee Y\to X$ the retraction which maps all of $Y$ to $x_0=y_0$, and $r_Y:X\vee Y \to Y$ the retraction which maps $X$ to $x_0=y_0$. Since $\gamma_X$ is a strong strategy, then $\gamma(t)=r_X \rho(t)$ for some $t\in [0,\frac{1}{2}]$. If $t\neq \frac{1}{2}$, then $\gamma(t)\neq x_0$, so $\rho(t)\in X$ and $\gamma(t)=\rho(t)$. If $t=\frac{1}{2}$, then $\rho(t)\in Y$. Similarly, if $$
%\end{proof}
Note that the property of having a strategy is not invariant under reversal of the order. Concretely the cop has a strategy in $S_{2,1}^{op}$ while the robber has a strategy in $S_{2,1}$.

Theorem \ref{main1} completes the classification of the finite $T_0$ spaces of height 1 in which the cop has a strategy. On the other hand, by Corollary \ref{coroe} if $X$ is any finite $T_0$ space and the robber has a strategy in $\E(X)$, so does he in $X$. What if the cop has a strategy in $\E (X)$? This does not mean that he has a strategy in $X$ as we will see.

%\section{The general case}

\begin{ej} \label{ye}
The robber has a strategy in the space $X=\mathfrak{Y}$ in Figure \ref{figygriega}.

\begin{figure}[h] 
\begin{center}
\includegraphics[scale=0.55]{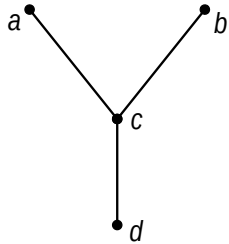}
\caption{The space $\mathfrak{Y}$.}\label{figygriega}
\end{center}
\end{figure}

Let $\gamma: [0,1]\to X$ be a path. Let $\mathcal{U}=\{U_a,U_b\}$. We consider a minimal $\gamma$-admissible subdivision of $[0,1]$. Suppose $[t_0,t_1], [t_1,t_2]$ are adjacent intervals in the subdivision. Without loss of generality we assume $U_a$ has been chosen for $[t_0,t_1]$ (so in particular $\gamma([t_0,t_1])\subseteq U_a$) and $U_b$ has been chosen for $[t_1,t_2]$. Since $U_a\cap U_b=\{c,d\}$ is open, there exist $s_0,s_1\in [0,1]$ with $\frac{t_0+t_1}{2}<s_0<t_1<s_1<\frac{t_1+t_2}{2}$ and $\gamma ((s_0,s_1))\subseteq \{c,d\}$. We will define $\rho\restr_{[\frac{t_0+t_1}{2},\frac{t_1+t_2}{2}]}$. If $d\notin \gamma((s_0,s_1))$, we define $\rho$ to be $d$ in $(s_0,s_1)$, $a$ in $[s_1,\frac{t_1+t_2}{2}]$, $b$ in $[\frac{t_0+t_1}{2},s_0]$. On the other hand, if $d\in \gamma((s_0,s_1))$, since $\{d\}$ is open in $X$, there exist $s_0<r_0<r_1<s_1$ such that $\gamma$ is identically $d$ in $(r_0,r_1)$. Then we define $\rho$ to be $c$ in $(r_0,r_1)$, $a$ in $[r_1,\frac{t_1+t_2}{2}]$, $b$ in $[\frac{t_0+t_1}{2},r_0]$. In any case, $\rho\restr_{[\frac{t_0+t_1}{2},\frac{t_1+t_2}{2}]}$ is continuous. We do this for every pair of adjacent intervals in the subdivision of $[0,1]$. As always, the definition of $\rho$ in the first interval of the subdivision and the last one is simpler and left to the reader. The resulting map $\rho:[0,1]\to X$ is continuous and has no coincidences with $\gamma$. 
\end{ej} 

\begin{coro} \label{coroy}
Let $X$ be a finite $T_0$ space which contains a subspace homeomorphic to $\mathfrak{Y}$, which we identify with $\mathfrak{Y}$. Suppose $\{a,b\}\subseteq \mathfrak{Y}$ has no upper bound in $X$. Then $\mathfrak{Y}$ is a retract of $X$, and in particular the robber has a strategy in $X$. 
\end{coro}
\begin{proof}
Define $r:X\to \mathfrak{Y}$ as follows. Let $x\in X$. If $x\ge a$, $r(x)=a$. If $x\ge b$, $r(x)=b$. If $x\le d$, $r(x)=d$. In any other case, $r(x)=c$. Note that $r$ is well defined by hypothesis. We check that $r$ is continuous. Let $x\le y$ be two points in $X$. If $x\le d$, then $r(x)=d\le r(y)$. If $x\ge a$, then $r(x)=a=r(y)$. Similarly, if $x\ge b$, $r(x)=r(y)$. In any other case $r(x)=c$ and $r(y)\neq d$, so $r(x)\le r(y)$. Since $r$ fixes $\mathfrak{Y}$ pointwise, it is a retraction.
\end{proof}

\begin{obs} \label{kevin}
In Corollary \ref{coroe} we proved that if the cop has a strategy in a finite $T_0$ space $X$, then so does he in $\E (X)$. We used Proposition \ref{prope}. Corollary \ref{coroy} can be used to give an alternative proof. If the cop has a strategy in $X$, then any non-minimal point is smaller than or equal to a unique maximal element. Otherwise, the hypotheses of Corollary \ref{coroy} would be fulfilled. In this case $\E (X)$ is a retract of $X$, with a retraction which maps every non-minimal point to the unique greater maximal point. Thus, the cop has a strategy in $\E (X)$. 

\end{obs}

From Corollary \ref{coroe}, Theorem \ref{main1} and Corollary \ref{coroy}, we deduce the following result.

\begin{coro} \label{maincoro}
If $X$ is a finite $T_0$ space in which the cop has a strategy, then 
\begin{itemize}
\item $\E (X)$ is simply-connected,
\item $S_{2,1}$ and $S_{3,0}^{op}$ are not subspaces of $\E(X)$,
\item There is no subspace $A$ of $X$ homeomorphic to $\mathfrak{Y}$ with $\E(A) \subseteq \E(X)$.
\end{itemize}
\end{coro}

We do not know if the converse of Corollary \ref{maincoro} holds. In the remaining of the paper we prove a converse in a particular case and analyze a delicate example.

\bigskip

\section{The max strategy}

\begin{teo} \label{teomax}
Let $X$ be a finite $T_0$ space with maximum $x$. Then the cop has a strategy in $X$. Moreover, there exists a loop $\zeta_X:[0,1]\to X$ at $x$ which has a coincidence with any other path in $X$ and, furthermore, for every neighborhood $U$ of $0$ in $[0,1]$ and every continuous map $\rho:U\to X$, there is a coincidence between $\zeta_X\restr_U$ and $\rho$. The map $\zeta_X$ has the following property: if $X$ has a minimum $y$, then $\zeta_X^{-1}(y)$ is dense in $[0,1]$.
\end{teo}
\begin{proof}
The proof is by induction on $n=\# X$. If $n=1$, the unique path $\zeta_X:[0,1]\to X$ satisfies all the requirements. Let $X$ be any finite $T_0$ space with maximum $x$. Let $x_0,x_1,\ldots ,x_{m-1}$ be the points covered by $x$. By induction there exist paths $\zeta_{U_{x_i}}:[0,1]\to U_{x_i}$ satisfying the conditions in the statement for every $0\le i\le m-1$. We define $\zeta=\zeta_X:[0,1]\to X$ as follows. As required $\zeta(0)=\zeta(1)=x$. For each $k\ge 2$, $\zeta (\frac{1}{k})=x_{\overline{k}}$, where $\overline{k}$ is the class of $k$ modulo $m$. Moreover $\zeta$ restricted to the interval $[\frac{1}{k}, \frac{1}{2}(\frac{1}{k}+\frac{1}{k-1})]$ is (the linear reparametrization of) $\zeta_{U_{x_{\overline{k}}}}$ for $k\ge 2$. Finally, $\zeta$ restricted to $[\frac{1}{2}(\frac{1}{k}+\frac{1}{k-1}), \frac{1}{k-1}]$ is a path in $X$ from $x_{\overline{k}}$ to $x_{\overline{k-1}}$ if $k\ge 3$, and a path from $x_{\overline{2}}$ to $x$ if $k=2$. If $X$ has a minimum $y$, we take this path to be constant with value $y$ in $(\frac{1}{2}(\frac{1}{k}+\frac{1}{k-1}), \frac{1}{k-1})$ (see Figure \ref{figmax}). Note that $\zeta$ is continuous in $0$ since $x$ is the maximum of $X$. Note also that if $X$ has minimum $y$, then $y$ is the minimum of every $U_{x_i}$, so by induction $\zeta^{-1}(y)$ is dense in $[0,1]$.

\begin{figure}[h] 
\begin{center}
\includegraphics[scale=0.55]{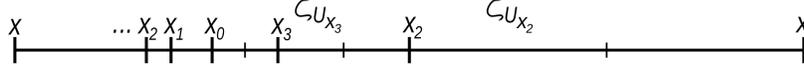}
\caption{The path $\zeta$ in the case that the maximum covers $m=4$ points.}\label{figmax}
\end{center}
\end{figure}
%\begin{figure}[h] 
%\begin{center}
%\includegraphics[scale=0.55]{}
%\caption{}\label{figygriega}
%\end{center}
%\end{figure}

Let $U$ be a neighborhood of $0$ in $[0,1]$ and let $\rho:U\to X$ be a continuous map. We want to prove that $\zeta\restr_U$ and $\rho$ have a coincidence point. If $\rho(0)=x$, then $\rho(0)=\zeta(0)$. Suppose then that $\rho(0)\neq x$. Then $\rho(0)\in U_{x_i}$ for some $0\le i\le m-1$. Since $U_{x_i}$ is open, there is a neighborhood $V$ of $0$ in $U$ mapped into $U_{x_i}$ by $\rho$. There exists $k\in \N$ such that $\overline{k}=i$ and $I=[\frac{1}{k}, \frac{1}{2}(\frac{1}{k}+\frac{1}{k-1})]\subseteq V$. Thus the restriction of $\zeta$ to $I$ is (the linear reparametrization of) $\zeta_{U_{x_i}}$ and therefore it has a coincidence with $\rho\restr_I$.
\end{proof}

\begin{defi}
Let $X$ be a finite $T_0$ space with maximum. Any path $\zeta_X:[0,1]\to X$ as constructed in the proof of Theorem \ref{teomax} will be called \textsl{the} \textit{max strategy} in $X$, even though it is not uniquely determined by $X$. Any linear reparametrization $[t_0,t_1]\to X$ of $\zeta_X$ will also be called the max strategy in $X$.
\end{defi}

\begin{teo} \label{main2}
Let $X$ be a finite $T_0$ space such that the following conditions are satisfied simultaneously:
\begin{enumerate}
\item $\E (X)$ is simply-connected,
\item $S_{2,1}$ and $S_{3,0}^{op}$ are not subspaces of $\E(X)$,
\item There is no subspace $A$ of $X$ homeomorphic to $\mathfrak{Y}$ with $\E(A) \subseteq \E(X)$.
\item There is no subspace $A$ of $X$ homeomorphic to $\mathfrak{Y}^{op}$ with $\E(A) \subseteq \E(X)$
\end{enumerate}
Then the cop has a strategy in $X$.
\end{teo}
\begin{proof}
The fact of $\E(X)$ being simply-connected is equivalent to $\E(X)$ containing no cycles.

Just as in the proof of Theorem \ref{main1}, we can assume $\E(X)$ looks as in Figure \ref{figgeneral}. We may assume $a_0$ and $a_l$ are minimal as well: adding points $a_{-1}$ and $a_{l+1}$ to $X$ if needed, as described in the proof of Theorem \ref{main1}, comparable with just one point of $X$ each, produces a new space satisfying all four conditions and having $X$ as a retract.

We define a path $\gamma: [0,l]\to X$. The first part of the definition looks very similar to the one of the strategy constructed in Theorem \ref{main1}, but then, in some intervals we will use the max strategy of certain subspaces of $X$. We define $\gamma(2i+1)=a_{2i+1}$ for $0\le i\le \frac{l}{2}-1$. The interval $[2i-1,2i+1]$ (for $1\le i \le \frac{l}{2}-1$) is subdivided in $m_{2i}+2$ intervals by taking points $2i-1<t_{2i,1}<t_{2i,2}<\ldots<t_{2i,m_{2i}}<t_{2i,m_{2i}+1}<2i+1$. We define $\gamma(t_{2i,j})=b_{2i,j}$ for $1\le j\le m_{2i}$. 

The intervals $[2i-1,t_{2i,1}]$ for $1\le i\le \frac{l}{2}-1$ are divided in two parts by taking a point $2i-1<r_{2i-1}<t_{2i,1}$. The intervals $[t_{2i,j},t_{2i,j+1}]$ for $1\le j\le m_{2i}$ are divided in two by taking $t_{2i,j}<r_{2i,j}<t_{2i,j+1}$. 
%The intervals $[t_{2i,m_{2i}}, 2i+1]$ for $1\le i\le \frac{l}{2}-1$ are divided in three by taking $t_{2i,m_{2i}}<r_{2i,m_{2i}}<s_{2i}<2i+1$.

In the interval $[0,1]$, $\gamma$ is the inverse path of the max strategy of $U_{a_1}$. In $[l-1,l]$, $\gamma$ is the max strategy of $U_{a_{l-1}}$. In the intervals $[t_{2i,j}, r_{2i,j}]$ for $1\le j\le m_{2i}$, $\gamma$ is the max strategy of $U_{b_{2i,j}}$. In $(r_{2i,j}, t_{2i,j+1})$ for $1\le j\le m_{2i}$, it is constant with value $a_{2i}$. In $[t_{2i,m_{2i}+1},2i+1]$ it is the inverse path of the max strategy for $U_{a_{2i+1}}\cap F_{a_{2i}}$. Here $F_x$ denotes the closure of $\{x\}$, i.e. the set of points which are greater than $x$. Finally, in $[2i-1,r_{2i-1}]$ it is the max strategy of $U_{a_{2i-1}}\cap F_{a_{2i}}$ while in $(r_{2i-1},t_{2i,1})$ it is the constant function $a_{2i}$. Clearly $\gamma:[0,l]\to X$ is continuous. We will prove that any other path $\rho:[0,l]\to X$ has a coincidence with $\gamma$.

Let $\rho:[0,l]\to X$ be a continuous map.

Let $Y$ be the subspace of $X$ formed by points $a_1,a_2, \ldots, a_{l-1}$. We define a retraction $r:X\to Y$ as follows. For $0\le i\le \frac{l}{2}-1$, all the points of $U_{a_{2i+1}}$ not in $Y$ are mapped to $a_{2i+1}$. For $1\le i\le \frac{l}{2}-1$ and $1\le j\le m_{2i}$, all the points of $U_{b_{2i,j}}$ are mapped to $a_{2i}$. The map $r$ is well-defined thanks to condition 3. It is clearly continuous.

Let $\omega:[1,l-1]\to Y$ be the map defined in the proof of Theorem \ref{main1}. By Remark \ref{obsintervalos} there exists $t\in [1,l-1]$ such that $\omega(t)=r\rho\restr_{[1,l-1]}(t)$.

Case 1: $t=1$. In this case $r \rho (1)=\omega (1)=a_1$, so $\rho (1)\in U_{a_1}$. Let $U=\rho^{-1}(U_{a_1})\cap [0,1]$. Then $U$ is a neighborhood of $1$ in $[0,1]$. By Theorem \ref{teomax}, $\rho$ and the inverse of $\zeta_{U_{a_1}}$ have a coincidence in $U$. Then so do $\rho$ and $\gamma$.

Case 2: $t=l-1$. This case is symmetric to case 1.

Case 3: $t= 2i-1$ for some $2\le i\le \frac{l}{2}-1$. In this case $r\rho(t)=\omega (t)=a_{2i-1}$, so $\rho(t)\in U_{a_{2i-1}}$. If $\rho (t)=a_{2i-1}$, then $\rho(t)=\gamma(t)$. Otherwise $\rho(t)<a_{2i-1}$. Then either $a_{2i-2}\le \rho(t)$ or $a_{2i}\le \rho (t)$. In the first case $U_{\rho(t)}\subseteq U_{a_{2i-1}}\cap F_{a_{2i-2}}$ by condition 4. Therefore $\rho$ and the inverse of $\zeta_{U_{a_{2i-1}}\cap F_{a_{2i-2}}}$ have a coincidence in $[t_{2i-2,m_{2i-2}+1},2i-1]$. In the second case $U_{\rho(t)}\subseteq U_{a_{2i-1}} \cap F_{a_{2i}}$, and $\rho$ and $\gamma$ have a coincidence in $[2i-1, r_{2i-1}]$. 

Case 4: $t\in (2i-1,2i+1)$ for some $1\le i\le \frac{l}{2}-1$. In this case $r \rho(t)=\omega(t)=a_{2i}$, so $\rho(t)=a_{2i}$ or $\rho(t)\in U_{b_{2i,j}}$ for some $j$. Note that by Theorem \ref{teomax}, $\gamma^{-1}(a_{2i})\cap[2i-1,2i+1]$ is dense in $[2i-1,2i-1]$, so in the first case $\rho$ and $\gamma$ have a coincidence in $[2i-1,2i+1]$. We can assume then that $\rho(t)\in U_{b_{2i,j}}$ for some $j$ and that $\rho(t')\neq a_{2i}$ for every $t'\in (2i-1,2i+1)$. By condition 3, $U_{b_{2i,j}}\smallsetminus \{a_{2i}\}$ is a connected component of $X\smallsetminus \{a_{2i}\}$, so $\rho((2i-1,2i+1))\subseteq U_{b_{2i,j}}\smallsetminus \{a_{2i}\}$. In particular $\rho$ and $\gamma$ have a coincidence in $[t_{2i,j}, r_{2i,j}]$. 

\end{proof}

\section{Self-similarity} \label{sectionfractal}

The fourth condition in the statement of Theorem \ref{main1} is not needed for the cop to have a strategy. This can be seen directly from the fact that the cop has a strategy in any space with maximum. In this section we describe an example in which the cop has a strategy, but any such strategy has to be more sophisticated than any we have constructed so far.

Self-similarity is a feature present in fractals, in which proper parts of an object look like the global object. This is a property our example will have.

Let $X$ be the space in Figure \ref{fractal}. We describe a strong strategy (in fact a path $\gamma: [-1,2]\to X$) for the cop in $X$. First we define $\sigma=\gamma \restr_{[0,1]}:[0,1]\to X$. We divide $[0,1]$ in four intervals of equal length. In the first interval $I_1=[0,\frac{1}{4}]$, $\sigma$ is defined as $d$ in the extremes and as $f$ in the interior. In the third interval $I_3=[\frac{1}{2},\frac{3}{4}]$, $\sigma$ is defined as $b$ in the left extreme, as $d$ in the right extreme and as $e$ in the interior. We also define $\sigma(1)=b$. This is stage 1 in the definition of $\sigma$. In the second and fourth intervals $I_2=[\frac{1}{4},\frac{1}{2}]$, $I_4=[\frac{3}{4},1]$, $\sigma\restr_{I_2}$ and $\sigma\restr_{I_4}$ are the linear reparametrizations of $\sigma:[0,1]\to X$. Concretely, to define $\sigma$ in $I_2$ we repeat the construction above: we divide $I_2$ in four intervals of equal length and define $\sigma$ in the first and third exactly as before. We do the same for $I_4$. Note that $\sigma(\frac{1}{4})$, $\sigma(\frac{1}{2})$, $\sigma(\frac{3}{4})$ and $\sigma(1)$ are well-defined. Now, $\sigma$ has been defined in all of $[0,1]$ with exception of four (open) intervals of length $\frac{1}{16}$. This is the end of stage 2. In general, in stage $n$ we define $\sigma$ in $2^n$ closed intervals of length $\frac{1}{4^n}$. We repeat this process indefinitely. Now $\sigma$ has been defined in all of $[0,1]$ with exception of a $G_{\delta}$ set $C$ of Lebesgue measure $0$. We will not need to know much about $C$. It is easy to see that in stage $n$ of the construction, $\sigma$ has been defined in the complement of a union of $2^n$ open intervals, and $t=\frac{1}{3}$ always lies in the first of these. Thus $\frac{1}{3}\in C$, and in fact it is the minimum of $C$. The supremum of $C$ is $1\notin C$. Note also that $\sigma^{-1}(f)\cap[0,\frac{1}{3}]$ is dense in $[0,\frac{1}{3}]$. We define $\sigma\restr_{C}$ to be the constant function $b$. We claim that $\sigma:[0,1]\to X$ is continuous. 

\begin{figure}[h] 
\begin{center}
\includegraphics[scale=0.65]{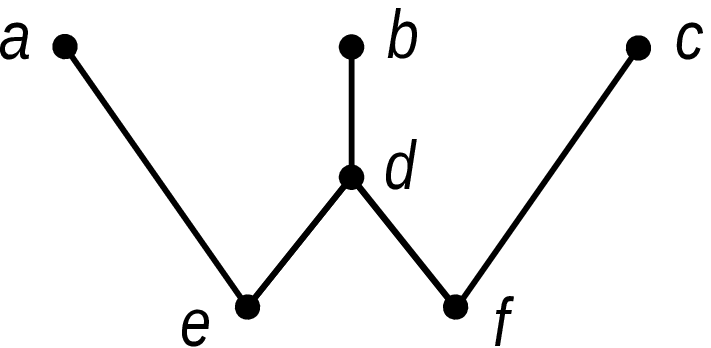}\\

\vspace{1cm}

\includegraphics[scale=0.55]{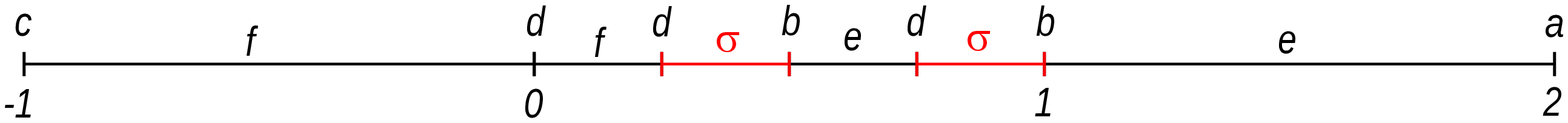}
\caption{The space $X$ and its fractal strategy $\gamma$.}\label{fractal}
\end{center}
\end{figure}

Indeed, let $t\in [0,1]$. We note that in each case there is a neighborhood of $t$ which is mapped inside $U_{\sigma(t)}$. If $\sigma (t)=b$, we can take the neighborhood as all $[0,1]$. If $\sigma(t)$ is equal to $e$ or $f$, then there is an open interval containing $t$ which is also mapped to $\sigma(t)$. If $\sigma (t)=d$, there is an open interval $I$ containing $t$ with $\sigma(I\smallsetminus \{t\})\subseteq \{e, f\}\subseteq U_d$.

Before we complete the definition of the strategy for the cop, we discuss some properties of $\sigma$. 

Property (i): if $t\in \sigma^{-1}(b)$, then any neighborhood $U$ of $t$ contains an interval $[s,s']$, which in turn contains a point $s''\in (s,s')$, such that $\sigma([s,s''))=\{e\}$, $\sigma(s'')=d$ and $\sigma((s'',s'])=\{f\}$. Indeed, by definition there is some $n\ge 1$ such that in stage $n$ one of the open intervals, say $J$, in which $\sigma$ has not been defined yet lies inside $U$. Thus, in stage $n+1$, the closure $\overline{J}$ of $J$ is divided in four closed intervals $J_1,J_2,J_3,J_4$ and $\sigma$ is defined in $J_1$ and $J_3$. Then, in stage $n+2$, $J_4$ is divided in four intervals $J_1',J_2',J_3',J_4'$ and $\sigma$ is defined in the first and the third one. If $J_3=[r,r']$ and $J_1'=[r',r'']$, the points $s=\frac{r+r'}{2},s'=\frac{r'+r''}{2},s''=r'$ satisfy the required condition. Note in particular that for any neighborhood $U$ of $t$, $\sigma(U)=U_b$. 

%Property **: if $t\in \sigma^{-1}(d)$, then there exists $t<t'<1$ such that $\sigma((t,t'))=\{f\}$. This is obvious by analysing the stage $n$ at which $\sigma(t)$ is defined and the definition of $\sigma$ in stage $n+1$.

Property (ii): If $r_0<r_1\in [0,1]$ and $\sigma(r_0)=f, \sigma(r_1)=e$, then there exists $r_0<r<r_1$ such that $\sigma(r)=b$. This is easy to prove by considering the stages $n,m$ in which $\sigma(r_0)$ and $\sigma(r_1)$ have been defined and analyzing the cases $n=m$, $n<m$ and $n>m$.

Property (iii): $\sigma^{-1}(\{e,f\})$ is dense in $[0,1]$. For $t\in [0,1)$ it is easy to see that if $\sigma (t)$ was defined in stage $n$, then any neighborhood of $t$ already contains a point $s$ such that $\sigma (s)$ was also defined in stage $n$ and it is equal to $e$ or $f$. For $t=1$, $\sigma(t)=b$ so the assertion follows directly from Property (i).

We define now the strong strategy $\gamma:[-1,2]\to X$ for the cop: $\gamma (-1)=c$, $\gamma\restr_{(-1,0)}$ is constant with value $f$, $\gamma\restr_{[0,1]}=\sigma$, $\gamma\restr_{(1,2)}$ is constant with value $e$, and $\gamma(2)=a$. Clearly $\gamma$ is continuous. Let $\rho:[-1,2]\to X$. We want to prove that there exists a point in which $\gamma$ and $\rho$ coincide. Assume there is no such coincidence.

Let $t\in [0,1]$ be such that $\sigma(t)=b$. We claim then that $\rho(t)\in \{a,c\}$. If $\rho(t)\in \{e,f\}$, there is a neighborhood $U$ of $t$ mapped to $\rho(t)$ by $\rho$. But since $\sigma(t)=b$, by Property (i), $\sigma(U\cap [0,1])=U_b \ni \rho(t)$. Thus, $\gamma$ and $\rho$ have a coincidence in $U\cap [0,1]$. If $\rho(t)=d$, there is a neighborhood $U$ of $t$ mapped to $U_d$ by $\rho$. By Property (i), there exists an interval $[s,s']\subseteq U$ and a point $s''\in (s,s')$ such that $\sigma\restr_{[s,s'')}$ is $e$, $\sigma(s'')=d$, and $\sigma\restr_{(s'',s']}$ is $f$. But then $\sigma\restr_{[s,s']}$ is a strong strategy in the space $U_d$, so it must have a coincidence with $\rho\restr_{[s,s']}$. Thus, $\rho(t)$ is equal to $a$ or $c$. 

In particular, when $t=1$, we deduce $\rho(1)=c$, since if $\rho(1)=a$, there would be a coincidence in $(1,2]$. Let $t_1=\min \{t\in [0,1] : \rho (t)=c\}$. We claim that $t_1>\frac{1}{3}$. Otherwise, since $\gamma^{-1}(f)\cap [-1,\frac{1}{3}]$ is dense in $[-1,\frac{1}{3}]$, $\rho\restr_{[-1,t_1]}$ could not go through $f$, so it should be constant, and we would have $\gamma(-1)=c=\rho(-1)$. Since $\frac{1}{3}<t_1$ and $\sigma(\frac{1}{3})=b$, then $\rho(\frac{1}{3})=a$. Let $t_0=\max \{t\in [0,t_1] : \rho(t)=a\}$.

Since $U_a=\{a,e\}$ and $U_c=\{c,f\}$ are open, by definition of $t_0$ and $t_1$ there exist $t_0<s_0<s_1<t_1$ such that $\sigma\restr_{(t_0,s_0)}$ is constant with value $e$ and $\sigma\restr_{(s_1,t_1)}$ is constant with value $f$. Since $\sigma^{-1}(\{e,f\})$ is dense in $[0,1]$ by Property (iii), and $\gamma$ and $\rho$  are coincidence-free, there exist $r_0\in (t_0,s_0)$ and $r_1\in (s_1,t_1)$ such that $\sigma (r_0)=f$ and $\sigma (r_1)=e$. By Property (ii) there exists $r\in (r_0,r_1)$ such that $\sigma (r)=b$. By Property (i), $\rho(r)\in \{a,c\}$, which contradicts the definition of $t_0$ or $t_1$.

%Since $U_c\subseteq X$ is open, $\rho^{-1}(U_c)$ is a disjoint union of open intervals, intersected with $[0,1]$. Let $(s_0,s_1)$ the interval among those which contains $t_1$. Then $s_0>t_0$ and by definition of $t_1$, $\rho |_{(s_0,t_1)}$ is constant $f$. Since $\rho(s_0)\notin U_c=\{f,c\}$ and $\rho(s_0)$ is in the closure of $\{f\}$, then $\rho(s_0)\in \{b,d\}$

%\begin{teo}\label{main3}
%Let $X$ be a finite $T_0$ space of height 2 such that the following conditions are satisfied simultaneously
%\begin{itemize}
%\item $\E (X)$ is simply-connected,
%\item $S_{2,1}$ and $S_{3,0}^{op}$ are not subspaces of $\E(X)$,
%\item There is no subspace $A$ of $X$ homeomorphic to $\mathfrak{Y}$ with $\E(A) \subseteq \E(X)$.
%\end{itemize}
%Then the cop has a strategy in $X$.
%\end{teo}
%\begin{proof}
%
%\end{proof}

\section*{Appendix}

\begin{teo} \label{admissible}
Let $X$ be a topological space, $\mathcal{U}$ a finite open cover of $X$ and $\gamma :\R_{\ge 0}\to X$ a curve. Then there exists a minimal $\gamma$-admissible subdivision of $\R_{\ge 0}$.
\end{teo}
\begin{proof}
By a Lebesgue number argument, each interval $[n,n+1]\subseteq \R_{\ge 0}$ ($n\in \Z_{\ge 0}$) can be subdivided into finitely many closed intervals, each of which is mapped by $\gamma$ into one of the members of $\mathcal{U}$. This is a $\gamma$-admissible subdivision which could be non-minimal. The poset of $\gamma$-admissible subdivisions of $\R_{\ge 0}$ is ordered by refinement. If $\{S_{\alpha}\}_{\alpha \in \Lambda}$ is a non-empty chain of $\gamma$-admissible subdivisions, then there is a lower bound defined as follows. Fix $\alpha_0\in \Lambda$. For each interval $I$ in $S_{\alpha_0}$ let $\widetilde{I}$ be the union of the intervals containing $I$ from among all the $S_{\alpha}$ with $\alpha \in \Lambda$. It is clear that $\widetilde{I}$ is an interval and that it is not a singleton. Let $t'\notin \widetilde{I}$. Then $t'$ is in the interior of an interval in $S_{\alpha_0}$ or in the boundary of two intervals in $S_{\alpha_0}$. In the first case, suppose $t'\in (a,b)$, with $[a,b]$ a member of $S_{\alpha_0}$ (or $t'=0$ and $a=0$). Then we claim that $(a,b)\cap \widetilde{I}=\emptyset$. Otherwise, there would be an interval $J$ in some $S_{\alpha}$ containing $I$ and intersecting $(a,b)$. Since $S_{\alpha_0}$ and $S_{\alpha}$ are comparable, either $[a,b]\subseteq J$ or $J\subseteq [a,b]$. In either case, $[a,b]\subseteq \widetilde{I}$, so $t'\in \widetilde{I}$, a contradiction. Similarly, if $t'$ is in the boundary of two intervals of $S_{\alpha_0}$, say $[a,t']$ and $[t',b]$, we claim that $(a,b)\cap \widetilde{I}=\emptyset$. Otherwise, there is an interval $J$ in some $S_{\alpha}$ containing $I$ and intersecting $(a,t')$ without loss of generality. As before this implies that $[a,t']\subseteq \widetilde{I}$, a contradiction. This proves that there is a neighborhood of $t'$ disjoint with $\widetilde{I}$, so $\widetilde{I}$ is closed. It is clear that the collection $S$ of all the $\widetilde{I}$ with $I\in S_{\alpha_0}$ covers $\R_{\ge 0}$. Suppose $\widetilde{I}$ and $\widetilde{J}$ have at least two points in common, say $a, b$, for some $I,J\in S_{\alpha_0}$. Using that the family of subdivisions is a chain, there is an interval in some $S_{\alpha}$ containing both $I$ and $J$, so $\widetilde{I}=\widetilde{J}$. This proves that the intersection of two different members of $S$ is empty or one point. Since $S_{\alpha_0}$ refines $S$, the later is locally finite. Finally, we show that $S$ is $\gamma$-admissible. Suppose $\widetilde{I}$ is not contained in $\gamma^{-1}(U)$ for any $U\in \mathcal{U}$. Then for every $U\in \mathcal{U}$ there exists $\alpha_U\in \Lambda$ and $J_U\in S_{\alpha_U}$ such that $I\subseteq J_U \subsetneq \gamma^{-1}(U)$. Let $V\in \mathcal{U}$ be such that $S_{\alpha_{V}}$ is the minimum among the (finitely many) $S_{\alpha_U}$. Then $J_V$ contains all the $J_U$, so $J_V\subsetneq \gamma^{-1}(U)$ for every $U\in \mathcal{U}$. This contradicts the fact that $S_{\alpha_V}$ is $\gamma$-admissible. It is clear that $S\le S_{\alpha}$ for every $\alpha \in \Lambda$. Indeed, if $J\in S_{\alpha}$, there is $I\in S_{\alpha_0}$ intersecting $J$ in at least two points, so either $I\subseteq J$ or $J\subseteq I$. In any case $J\subseteq \widetilde{I}$, so $S_{\alpha}$ is a refinement of $S$. By Zorn's Lemma there is a minimal element in the poset of $\gamma$-admissible subdivisions of $\R_{\ge 0}$, which is then a minimal subdivision.  
\end{proof}

\textbf{Acknowledgement}: I would like to thank Kevin Piterman for valuable suggestions.


\begin{thebibliography}{99}

\bibitem{olimpiada} \textit{III Olimpiada Iberoamericana Universitaria de Matem\'atica}, 2000.

\bibitem{ABG} S. Alexander, R.L. Bishop, R. Ghrist. \textit{Pursuit and evasion in non-convex domains
of arbitrary dimension}.
In Proceedings of Robotics: Science and Systems, Philadelphia,
USA (2006).

\bibitem{Ale} P.S. Alexandroff. \textit{Diskrete r\"aume}.
    MathematiceskiiSbornik (N.S.) 2 (1937), 501-518.

\bibitem{Bar} J.A. Barmak. \textit{Algebraic topology of finite topological spaces and applications}.
		Lecture Notes in Mathematics Vol. 2032. Springer (2011) xviii+170 pp.
		
\bibitem{Bar1} J.A. Barmak. \textit{Lion and man in non-metric spaces}. Math. Z. 290 (2018). 1165-1172.	

\bibitem{Bar3} J.A. Barmak. \textit{The fixed point property in every weak homotopy type}. American J. of Math. 138 (2016), 1425-1438.

\bibitem{BMW} J.A. Barmak, M. Mrozek, T. Wanner. \textit{Conley index for multivalued maps on finite topological spaces}. Preprint.	

\bibitem{BLW} B. Bollob\'as, I. Leader, M. Walters. \textit{Lion and man-can both win?} 
Israel J. Math. 189 (2012), 267-286.

\bibitem{BKK} M. Bramson, K. Burdzy, W. Kendall. \textit{Shy couplings, $\cat(0)$ spaces, and the lion and man}.  Ann. Probab. 41 (2013), 744-784.

\bibitem{BKK2} M. Bramson, K. Burdzy, W. Kendall. \textit{Rubber bands, pursuit games and shy couplings}.
Proc. London Math. Soc. 109 (2014), 121-160.

\bibitem{BW} J.R. Britnell, M. Wildon. \textit{Finding a princess in a palace: a pursuit–evasion problem}. Electronic J. Combin. 20 (2013).

\bibitem{Croft} H.T. Croft. \textit{Lion and man: a postscript}.
J. London Math. Soc. 39 (1964) 385-390.

%\bibitem{GR} C. Godsil, G. Royle. \textit{Algebraic graph theory}.

\bibitem{Has} J. Haslegrave. \textit{An evasion game on a graph}. Discrete Math. 314 (2014), 1-5.

\bibitem{HN} P. Hell, J. Ne\v set\v ril. \textit{Graphs and homomorphisms}. Oxford University Press, 2004.

\bibitem{Lit} J. E. Littlewood. \textit{A mathematician's miscellany}.
Methuen \& Co., Ltd., London, 1953. vii+136 pp. 

\bibitem{LW} A.T. Lundell, S. Weingram. \textit{The Topology of CW Complexes}. Van Nostrand, New York, 1969.

\bibitem{Mcc} M.C. McCord. \textit{Singular homology groups and homotopy groups of finite topological spaces}.
    Duke Math. J. 33 (1966), 465-474.
		
\bibitem{Mun} J.R. Munkres. \textit{Topology, Second Edition}.
		Prentice Hall (2000).
		
\bibitem{Sga} J. Sgall. \textit{A solution of David Gale’s lion and man problem}.
Theoretic Comput. Sci. 259 (2001), 663-670.		
		
		
\bibitem{Sto} R.E. Stong. \textit{Finite topological spaces}.
    Trans. Amer. Math. Soc. 123 (1966), 325-340.
		
\bibitem{Wof} E. Wofsey. \textit{On the algebraic topology of finite spaces}. 2008. https://www. docdroid.net/zJI3nw2/finite-spaces-wofsey-pdf	

        
\end{thebibliography}
\end{document}